\documentclass{article}

\usepackage{cite}
\usepackage{amsmath}
\usepackage{amssymb}
\usepackage{amsthm}
\usepackage{algorithm}
\usepackage[noend]{algpseudocode}
\usepackage{url}
\usepackage{multicol}
\usepackage[margin=1in]{geometry}
\usepackage{layouts}
\usepackage[nolist]{acronym}
\usepackage{parskip}
\usepackage[font=small,labelfont=bf]{caption,subcaption}
\usepackage{booktabs}
\usepackage{tikz}
\usetikzlibrary{calc,arrows.meta,shapes,positioning,decorations.markings}
\tikzstyle{myPath} = [-{Latex[length=3mm,width=2mm]}, line width=0.3mm]
\tikzstyle{myPath2} = [-{Latex[length=2mm,width=1mm]}, line width=0.2mm]
\usepackage{amssymb}

\newcommand{\bmat}[1]{\begin{bmatrix}#1\end{bmatrix}}

\newtheorem{thm}{Theorem}[section]
\newtheorem{lem}[thm]{Lemma}

\newtheorem{prop}{Proposition}

\newtheorem{rem}{Remark}

\usepackage{xcolor}

\newcommand{\bt}[1]{\textbf{#1}}

\begin{document}

\title{Two-stage stochastic approximation for dynamic rebalancing of shared mobility systems}

\author{Joseph Warrington\thanks{Simons Institute for the Theory of Computing, UC Berkeley, 121 Calvin Lab, Berkeley, CA 94720-2190, USA; Automatic Control Lab, ETH Zurich, Physikstrasse 3, 8092 Zurich, Switzerland. E-mail: \texttt{warrington@control.ee.ethz.ch}.}
~and Dominik Ruchti
\thanks{Automatic Control Lab, ETH Zurich, Physikstrasse 3, 8092 Zurich, Switzerland.}}

\maketitle


\begin{abstract}
Mobility systems featuring shared vehicles are often unable to serve all potential customers, as the distribution of demand does not coincide with the positions of vehicles at any given time. System operators often choose to reposition these shared vehicles (such as bikes, cars, or scooters) actively during the course of the day to improve service rate. They face a complex dynamic optimization problem in which many integer-valued decisions must be made, using real-time state and forecast information, and within the tight computation time constraints inherent to real-time decision-making. We first present a novel nested-flow formulation of the problem, and demonstrate that its linear relaxation is significantly tighter than one from existing literature. We then adapt a two-stage stochastic approximation scheme from the generic SPAR algorithm due to Powell \textit{et al.}, in which rebalancing plans are optimized against a value function representing the expected cost (in terms of fulfilled and unfulfilled customer demand) of the future evolution of the system. The true value function is approximated by a separable function of contributions due to the rebalancing actions carried out at each station and each time step of the planning horizon. The new algorithm requires surprisingly few iterations to yield high-quality solutions, and is suited to real-time use as it can be terminated early if required. We provide insight into this good performance by examining the mathematical properties of our new flow formulation, and perform rigorous tests on standardized benchmark networks to explore the effect of system size. We then use data from Philadelphia's public bike sharing scheme to demonstrate that the approach also yields performance gains for real systems.
\end{abstract}
\begin{acronym}
	\acro{ADMM}{alternating direction method of multipliers}
	\acro{DDP}{Dual dynammic programming}
	\acro{DP}{Dynammic programming}
    \acro{DRRP}{Dynamic Repositioning and Routing Problem}
    \acro{LP}{linear programming}
    \acro{LR}{Lagrangian relaxation}
    \acro{MIP}{mixed-integer program}
    \acro{MPC}{model predictive control}
    \acro{NYC}{New York City}
    \acro{RV}{rebalancing vehicle}
    \acro{SV}{shared vehicle}
    \acro{VF}{value function}
\end{acronym}


\section{Introduction}
Shared mobility systems, in which vehicles such as bicycles, scooters, and cars are used on demand by customers for a point-to-point journey, are an integral part of many transportation networks. A substantial fraction of cities worldwide have public bike sharing schemes, and at time of writing a recent trend has emerged toward so-called ``station free'' services featuring bicycles, scooters or cars that can be unlocked via smartphone apps and left in any legal public location after customers complete journeys.

Operators of such schemes, be they city franchise holders or autonomous private companies, typically aim to maximize the number (or total value) of journeys the system is able to support, particularly in cases where customers pay a flat annual fee to access the service. Since customer demand is not evenly distributed over the network, improving the service rate typically implies paying staff to redistribute the shared vehicles manually. However, this incurs costs that must be traded off against any improvement to service quality. Some existing systems also allow private agents to participate in the rebalancing effort alongside the operator by offering appropriate incentives, e.g.~the Bike Angels program\footnote{\url{https://www.citibikenyc.com/bikeangels/}, accessed September 2018} for Citibike in New York City, and the Charger program\footnote{\url{https://www.bird.co/charger}, accessed September 2018} offered by Bird.

Focusing on the system operator's role, the decision of \emph{how} to rebalance the system, meaning which staff movements and redeployments of shared vehicles should take place and when, is complicated by the random nature of customer demand. In many modern systems, customers rent vehicles spontaneously without reservation, and it is not known exactly where they will go or how long they will take to get there. Therefore the operator can only estimate their future behaviour in a statistical manner, e.g.~with a model built from past data. Nevertheless, concrete rebalancing actions must be planned at any given time amidst this uncertainty; typically the operator wishes these to be optimal \emph{in expectation} over the full distribution of possible demand-side requests. If rebalancing actions are planned based only on a single assumed demand scenario, performance may turn out to be poor in other, equally credible scenarios. Thus demand uncertainty is arguably the most material source of difficulty in the system operator's decision problem, with potential staff holdups caused by road congestion being an additional factor to consider.

Some of the operator's rebalancing actions can be planned or carried out ahead of time, for example by acting at night while the system is closed \cite{raviv_optimal_2013}, or by using a decision rule that is a function of the current system state but which does not itself change over time \cite{jorge_comparing_2014}. However, operators can better respond to the dynamically-evolving state of the network if they use an \emph{online optimization} approach \cite{contardo_balancing_2012, ghosh_dynamic_2017}, solving what is known in \cite{ghosh_dynamic_2017} as the \ac{DRRP}. This optimization, typically computing a look-ahead plan a number of hours into the future, tailors the decision problem to current needs and can exploit short-term forecast information, for example an estimate of changes in demand due to the weather, or a public event taking place in the city. It is therefore attractive to develop algorithms for solving such online problems quickly enough for a real-time implementation, meaning in minutes rather than hours, and a body of literature has emerged on the topic. 

Existing work on \ac{DRRP}-type problems, including formulations featuring autonomous vehicles and on-demand ride sharing, falls into two broad categories in terms of theoretical motivation, namely optimization-based and queuing-theoretic. Optimization-based methods directly model the system's costs, constraints, dynamics and demand data, and return an explicit system balancing decision, which may also include pricing decisions. Nair and Miller-Hooks \cite{nair_fleet_2011} considered the short-term problem of redistributing \acp{SV} as a chance-constrained integer program, in which demand rates were assumed constant and the operator wishes to serve most demand across rental locations with high probability. Ghosh \emph{et al.}~\cite{ghosh_dynamic_2017} used \ac{LR} to optimize routing and redeployment actions for a bike-sharing scheme against a pricing function representing the value of bikes at different nodes and different times. They encountered a common trade-off of \ac{LR} for large-scale problems, between the complexity of solving the decomposed subproblems and the quality of the solution once the algorithm has converged \cite{fisher_applications_1985}. They clustered the stations in the Boston test system in order to limit the computational complexity of each iteration and obtain a solution with acceptable accuracy; nevertheless the total computation time for a city-scale problem was on the order of a day. Pfrommer \textit{et al.}~used deterministic-equivalent forecasts of customer behaviour, and assumptions on their price sensitivity, to construct simultaneous schedules of system rebalancing actions and real-time prices for users \cite{pfrommer_dynamic_2014}. Recent literature on one-way car sharing systems also provides a number of formulations that fall into this category. Jorge \textit{et al.}~developed a simultaneous design and operational model for such systems which employed rebalancing heuristics and deterministic demand to simplify the problem \cite{jorge_comparing_2014}. Boyac{\i} \textit{et al.}~created a simultaneous routing, rebalancing, and charging formulation, and simplified the routing problem using dummy nodes to achieve a tractable integer programming formulation \cite{boyaci_optimization_2015}. The work was later elaborated to include a more sophisticated model of the rebalancing personnel's movements \cite{boyaci_integrated_2017}. Nourinejad \textit{et al.}~formulated a similar problem as a coupled travelling salesman problem for vehicles and personnel, and used an iterative heuristic to solve a problem featuring deterministic demand. Other recent studies have focused on real-time matching algorithms to allocate customers to vehicles at city scale in a ride-sharing context \cite{masoud_real-time_2017, alonso-mora_-demand_2017}.

Recent examples of optimization-based control have employed so-called commodity flow formulations borrowed from the network flow literature \cite{ahuja_network_1993}. In this setting, customers and shared vehicles are treated as fluid flows along capacitated links, where the links represent roads and vehicles carrying vehicles or passengers. Fan \textit{et al.}~presented an early formulation for car sharing systems, in which a stochastic optimization problem was approximated using Monte Carlo scenario generation \cite{fan_carsharing:_2008}. A problem with five time steps and four spatial nodes was solved. Contardo \textit{et al.}~solved a deterministic bike-share balancing problem using a flow-based formulation, employing Benders' decomposition and column generation to solve it \cite{contardo_balancing_2012}. Very recently, flow formulations have been used to address broader design questions for (autonomous) mobility systems, but these differ from the examples above in that they view all transport flows as continuous variables \cite{salazar_interaction_2018, rossi_routing_2018}. The chief attractions of flow formulations are that they can either assume away integrality constraints owing to the scale of the problem, or otherwise harness long-known results from single- or multi-commodity flow theory when integrality constraints are important.

In contrast to optimization, queuing models of mobility systems leverage theory designed for stochastic queuing processes, and are therefore amenable to pricing or other control policies accounting for this randomness \cite{regue_proactive_2014, banerjee_pricing_2015, schuijbroek_inventory_2017}. Inventory models such as \cite{raviv_optimal_2013} are inspired by mathematically similar problems of this kind from the operations research literature, and lead to a similar search for optimal \emph{policies}, rather than explicit control actions. He \textit{et al.}~embedded a queuing model of an electric car sharing system within an integer decision problem selecting which parts of a city the system should serve in order to maximize profit \cite{he_service_2017}. A drawback to queuing models, however, is that time-varying input data is significantly harder to model than in a discrete-time optimization setting -- although new preliminary results in this direction are emerging \cite{banerjee_value_2018}. 

The present paper addresses large-scale \ac{DRRP} instances with stochastic customer demand, within a discrete-time stochastic optimization setting. We propose a novel method based on stochastic approximation \cite{borkar_stochastic_2009}, in which a finitely-parameterized function is used to represent a function from a broader class as accurately as possible (according to a specifically defined metric). This approach is commonplace in reinforcement learning literature, where ``optimal'' coefficients for a vector of pre-selected basis function are learned from sampled behaviour of the system \cite{busoniu_reinforcement_2010}. In the case of shared mobility, we are interested in estimating the expected value of all customer journeys, under random demand patterns, as a function of the operator's rebalancing decisions. A good estimate of this \ac{VF} allows the operator to optimize its actions without needing to include the high dimensional and stochastic second stage in the same problem.

One approach to obtaining an accurate \ac{VF} would be to estimate, for example using Monte Carlo sampling of customer demand scenarios, the expected value of \emph{each decision} the system operator could make, and then pick the one that leads to the lowest cost. However, although the search space is enumerable -- the operator only has to decide where and when (in discretized time) to pick up and deposit a finite number of shared vehicles -- the number of possible actions is exponential in the number of spatial locations and time steps. Powell \textit{et al.}~suggest approximating the high-dimensional \acp{VF} encountered in many operations research problems as a separable sum of functions of their coordinate variables \cite{powell_learning_2004}. They describe a stochastic gradient scheme known as \emph{SPAR} to achieve this, and report good performance for two-stage stochastic programs. In other studies such as \cite{powell_stochastic_2003} the same authors applied several related approaches to logistics problems.

In this paper we adapt the SPAR algorithm to \ac{DRRP}. We cast the rebalancing problem as a two-stage stochastic program, where the operator makes rebalancing decisions in the first stage without knowledge of the realization of customer demand, which happens in the second stage. The \ac{VF} of the second stage is approximated as a sum of contributions from rebalancing actions across spatial locations and across time steps, or in other words, as the sum of values of a net change in \acp{SV} at a particular place and a particular time. Crucially, the number of parameters needed for this grows only linearly in the number of coordinate variables -- number of stations times number of time steps -- rather than exponentially. By approximating the second-stage \ac{VF} in its own right, we separate the difficulties of the first and second stages. The first stage is difficult because it requires the solution of an integer routing and redeployment problem, while the second stage is stochastic in nature and thus is different for each possible realization of demand. Because the true \ac{VF} is agnostic to the way first-stage decisions are made, it can be approximated in its own right as an intermediate step, and then inserted as a cost-to-go in the first stage, which is then solved to generate a high-quality solution to the overall problem. In this paper, our approach therefore has two steps, \textit{(i)} approximating the second-stage \ac{VF}, possibly using a simplified model of the first stage, and \textit{(ii)} re-solving the first-stage problem using the approximate \ac{VF} as input data.

The notion of a separable \ac{VF} has been used in other shared mobility settings. Raviv \textit{et al.}~\cite{raviv_optimal_2013} proved that, if coupling between stations of a bike sharing scheme is ignored, the cost of unsatisfied demand can be written as a convex function of the inventory levels of each station. In the \emph{static} rebalancing problem, the operator must arrange the system in an optimal configuration for service rate, but does not intervene after customers start using the system. Separable \acp{VF} were studied in \cite{raviv_static_2013} for optimizing the overnight ``reset'' of a bike-sharing scheme before customers start using it in the morning. Pal and Zhang extended the static problem to station-free bike sharing systems in \cite{pal_free-floating_2017}. The static problem is substantially different from our dynamic setting, because customer actions do not occur simultaneously alongside the operator's actions. Separable functions were also used in a dynamic, deterministic-equivalent setting in \cite{pfrommer_dynamic_2014} for the problem of offering customer incentives to make modified journeys that are favourable to the system's service rate. Legros \cite{legros_dynamic_2018} has also computed station \acp{VF} and associated Bernoulli policies for dynamic bike sharing, but without explicit predictive modelling of the kind we include here. Brinkmann \textit{et al.}~used non-parametric \ac{VF} estimates, in a similar dynamic look-ahead context to ours, but without considering vehicle routing in the decision problem \cite{brinkmann_dynamic_2018}.

\subsection{Model scope}

We present a general formulation in which randomly-appearing customers rent generic \acp{SV} and the system operator uses generic \acp{RV}, which are able to hold one or more \acp{SV}, to rebalance the system. This is a flexible model that accommodates a wide range of different shared mobility systems of interest, for example:
\begin{enumerate}
\item Conventional bike sharing services, in which nodes of the \ac{SV} network represent docking stations, \acp{RV} represent trucks able to carry a finite number of bikes between docking stations, and each docking station has a finite capacity.
\item Station-free shared bike services such as LimeBike, Mobike, and Ofo, in which the number of ``notional stations'' tends to infinity as continuous space is divided into increasingly fine partitions, with a model node associated with each. Since such bikes can be left freely in public spaces, the capacity of the notional stations can be considered unbounded. The computational cost of solving the resulting decision problem will of course increase as space is divided into finer partitions.
\item Shared e-scooter services such as eCooltra and Bird. In the case of ``micro'' scooters, these distinguish themselves from bikes in that, owing to their small size, large numbers can in principle be repositioned by a single manned vehicle, i.e., no upper bound on the \acp{RV} capacity is necessary.
\item One-way car sharing services such as Zipcar Flex and Car2Go, in which only one \ac{SV} (rental car) can be transported at a time by a single staff member. In this setting, \ac{RV} movements in the model correspond to the movements of \emph{staff}, rather than trucks that can contain \acp{SV}.
\end{enumerate}

\subsection{Contributions}

Formally, we make the following contributions:
\begin{enumerate}
\item We provide a network flow formulation of \ac{DRRP} whose optimal value is the same as the formulation of \cite{ghosh_dynamic_2017}, but which has a tighter linear relaxation (meaning at least as tight, and often much tighter);
\item We derive a two-stage approximation scheme adapted from \cite{powell_learning_2004} that accommodates stochastic demand and journey values, and uses a separable \ac{VF} to represent the costs of customer actions as a function of the system operator's rebalancing actions;
\item We provide insight into the structure of the two-stage formulation, showing that both stages are closely related to network flow problems with integer capacity constraints;
\item We show that this scheme yields highly satisfactory system performance within computation times suitable for real-time implementation. We demonstrate the method using rigorous tests on standardized networks, and a case study using public data from Philadelphia's public bike sharing scheme.
\end{enumerate}

The paper is structured as follows. Section \ref{sec:PS} describes the flow formulation of \ac{DRRP}, and derives some basic results concerning its decomposition into two stages. Section \ref{sec:SA} outlines the iterative value function approximation used to solve the problem. Section \ref{sec:NR} presents numerical results for synthetic systems and a model of the Philadelphia network, and Section \ref{sec:C} concludes.

\section{Problem statement} \label{sec:PS}

In \ac{DRRP}, the operator must plan routes and load/unload schedules for a fleet $\mathcal{V}$ of operator vehicles over $T$ discrete time steps, in order to minimize a generalized cost function which includes penalties for unserved customer demand, costs for the movements of the operator's \acp{RV}, and the cost of loading and unloading \acp{SV} into and out of \acp{RV}.

\subsection{Definitions} \label{sec:defs}

Let $\mathcal{G}_\text{SV} := (\mathcal{N}_\text{SV},\mathcal{E}_\text{SV})$ be a directed graph describing all possible customer movements while using a \ac{SV}, such that movement between from node $i$ to $j$ in $\mathcal{N}_\text{SV}$ is only possible if $(i,j) \in \mathcal{E}_\text{SV}$. Let $f_{i,j}^{t,k}$ be the number of customers wishing to travel from node $i$ to node $j$ at time $t$, and with journey duration $k \in \{0,\ldots,K\}$ time steps for some maximum modelled duration $K$. We will use the tuple $(i,j,t,k)$ as shorthand, and we view $\mathcal{G}_\text{SV}$ and $f_{i,j}^{t,k}$ as fixed input data to the decision problem. 

Let the decision variable $w_{i,j}^{t,k} \in \mathbb{N}$ represent the number of journeys that do take place, and let the loss function $l_{i,j}^{t,k} : \mathbb{N} \rightarrow \mathbb{R}$ represent the social cost of potential journeys that do not occur, such that $l_{i,j}^{t,k}(f_{i,j}^{t,k} - w_{i,j}^{t,k})$ is the cost of unserved demand for tuple $(i,j,t,k)$. We assume that customer journey values are decoupled within and between such tuples, and that $l_{i,j}^{t,k}$ is defined to penalize journeys in increasing order of their value. Thus, $l_{i,j}^{t,k}(\cdot)$ can be represented as a convex piecewise affine function passing through the origin, evaluated only on its integer arguments.

We introduce stochasticity into the model by allowing scalars $f_{i,j}^{t,k}$ and functions $l_{i,j}^{t,k}(\cdot)$ to depend on a random variable $\xi$ sampled from a probability distribution with support $\Xi$. All other input data to the problem remains deterministic. Thus we have $f_{i,j}^{t,k} : \Xi \rightarrow \mathbb{N}$ and $l_{i,j}^{t,k} : \mathbb{N} \times \Xi \rightarrow \mathbb{R}$. We use $\mathbb{E}[\cdot]$ to denote expectation over $\xi$. In the remainder of this paper we use the variants $f_{i,j}^{t,k}$, $f_{i,j}^{t,k}(\xi)$, $l_{i,j}^{t,k}(\cdot)$, and $l_{i,j}^{t,k}(\cdot; \xi)$ according to whether or not $\xi$ is relevant to the discourse.

Let $\mathcal{G}_\text{RV} := (\mathcal{N}_\text{RV},\mathcal{E}_\text{RV})$ be a directed graph describing possible movements of a set $\mathcal{V}$ of \acp{RV}, such that a \ac{RV} can move from node $i$ to $j$ in $\mathcal{N}_\text{RV}$ only if $(i,j) \in \mathcal{E}_\text{RV}$. We assume that $\mathcal{G}_\text{RV}$ is fixed problem data, and has been constructed such that each \ac{RV} movement always takes exactly one time step.\footnote{Within the scope of discrete-time models, this choice is without loss of generality, and primarily for computational convenience. Where necessary to reflect real-world conditions, additional nodes and edges can be used to model \ac{RV} journeys taking multiple time steps.}  Let $z_{i,j}^t \in \mathbb{Z}$ be the integer decision variable specifying how many \acp{RV} in $\mathcal{V}$ should move from node $i$ to node $j$ at time $t$. The associated cost per \ac{RV} is denoted $c_{i,j}^t$.

\acp{RV} are used to reposition \acp{SV} by loading and unloading them at nodes in $\mathcal{N}_\text{RV} \cap \mathcal{N}_\text{SV}$, and travelling with them on board. For the problem to be non-trivial, we must assume that $\mathcal{N}_\text{RV} \cap \mathcal{N}_\text{SV} \neq \emptyset$ so that \acp{RV} are able to access at least some of the nodes in $\mathcal{N}_\text{SV}$. For simplicity we assume the \acp{RV} are identical, and that each is capable of carrying some integer number $\overline{b}$ of \acp{SV}.

At a time step $t$, integer variables $y_{i}^{+,t}$ and $y_{i}^{-,t}$ represent the loading and unloading of \acp{SV} onto and from \acp{RV} respectively, for $i \in \mathcal{N}_\text{RV} \cap \mathcal{N}_\text{SV}$. These loading and unloading actions incur a cost of $r_i^t$ per \ac{SV}. The number of \acp{SV} travelling along edge $(i,j) \in \mathcal{E}_\text{RV}$ inside \acp{RV} is denoted $b_{i,j}^t$, and this is bounded by the sum of capacities of the \acp{RV} moving on the same edge at the same time, i.e.~by $\overline{b}z_{i,j}^t$. The number of \acp{SV} at node $i \in \mathcal{N}_\text{SV}$, \emph{after} accounting for any journeys finishing during time step $t$, is denoted $d_i^t \in \mathbb{Z}$. The number of \acp{SV} that can be stored at any node (or ``station'') $i \in \mathcal{N}_\text{SV}$ is denoted $\overline{d}_i$.

The main items of notation introduced in this section are listed in Table \ref{tab:notation}, and the interaction of the graphs $\mathcal{G}_\text{SV}$ and $\mathcal{G}_\text{RV}$ is depicted in Fig.~\ref{fig:twographs}.

\begin{table}[t]
\caption{Notation introduced in Section \ref{sec:PS}}
{\small
\begin{center}
\begin{tabular}{|ll|ll|}
\hline
SV/RV & Shared vehicle/rebalancing vehicle & \multicolumn{2}{|l|}{\bt{Sets}} \\
&  & $(\mathcal{N}_\text{SV}, \mathcal{E}_\text{SV})$ & Nodes and edges of \ac{SV} graph $\mathcal{G}_\text{SV}$\\
\multicolumn{2}{|l|}{\bt{Stochastic program}} & $(\mathcal{N}_\text{RV}, \mathcal{E}_\text{RV})$ & Nodes and edges of \ac{RV} graph $\mathcal{G}_\text{RV}$\\
$\xi$ & Uncertainty realization & $\mathcal{V}$ & Set of \acp{RV}  \\
$V(y; \xi)$ & Second-stage value function under $\xi$ & $\Xi$ & Support of uncertainty realizations\\
& & & \\
\multicolumn{2}{|l}{\bt{Optimization variables}} & \multicolumn{2}{|l|}{\bt{Fixed parameters}} \\
\textbullet${}_{i,j}^{t,k}$ & Quantity referring to: & $l_{i,j}^{t,k}(\cdot; \xi)$ & Loss function for unmet demand \\
& \hspace{1cm}start time $t$, duration $k$ steps, & $f_{i,j}^{t,k}(\xi)$ & Demand for \ac{SV} journeys \\
& \hspace{1cm}start node $i$, end node $j$& $c_{i,j}^t$ & Cost of \ac{RV} journey   \\
$w_{i,j}^{t,k}$ & Completed \ac{SV} journeys & $r_i^t$ & Cost of \ac{SV} loading or unloading at node $i$\\
$z_{i,j}^t$ & Planned \ac{RV} journey & $r_p$ & Model penalty for violating \ac{SV} conserv.\\
$y_i^{+,t},y_i^{-,t}$  & Loading of \acp{SV} onto, from \acp{RV} at node $i$ & $\overline{y}$ & Largest (un)loading action considered\\
$d_i^t$ & Number of \acp{SV} present at node $i$ & $\overline{d}_i$ & Capacity for \acp{SV} at node $i$ \\
$b_{i,j}^t$ & Flow of \acp{SV} on \ac{RV} graph & $\overline{b}$ & Capacity of any \ac{RV} to carry \acp{SV}\\
$p_i^{+,t},p_i^{-,t}$ & \acp{SV} added/subtracted for model feasibility & & \\
\hline
\end{tabular}
\end{center}
}
\label{tab:notation}
\end{table}%

\begin{figure}[htbp]
\begin{center}
\includegraphics[width=0.6\textwidth]{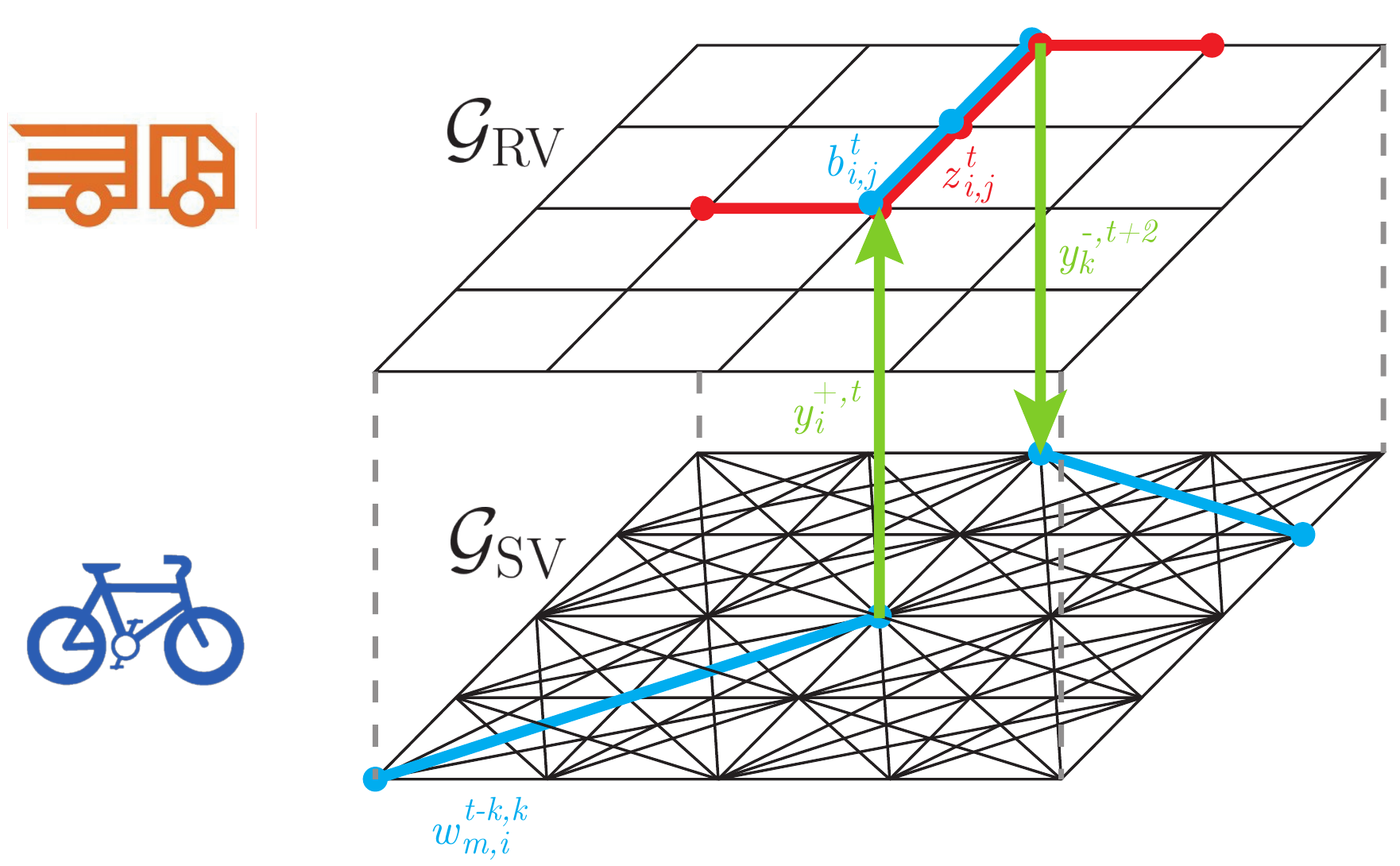}
\caption{Interaction of the graphs $\mathcal{G}_\text{SV}$ and $\mathcal{G}_\text{RV}$, with illustrative examples of indexed variables. While the \ac{SV} graph contains edges for journeys taking any number of time steps, the edges on the \ac{RV} graph are defined by movements that take precisely one time step. \ac{SV} movements are depicted in blue, \ac{RV} movements in red, and \ac{SV} load/unload actions in green. \ac{SV} movements can only occur on $\mathcal{G}_\text{RV}$ if accompanied by \acp{RV}.}
\label{fig:twographs}
\end{center}
\end{figure}

\subsection{Assumptions}

We make the following modelling assumptions:
\begin{enumerate}
\item[(A1)] A customer does not enter the system if (a) he cannot start his journey due to a lack of \acp{SV} available at his desired starting node, or if (b) doing so would cause the end node to exceed its \ac{SV} capacity. He does not attempt to start a journey from a different node. This simplifying assumption is discussed further after introducing our formulation, in Section \ref{sec:AddAss1}.
\item[(A2)] Journey times along edges in $\mathcal{E}_\text{RV}$ and $\mathcal{E}_\text{SV}$ are deterministic and independent of any of the actions we model, i.e., they are determined entirely by exogenous factors such as signals at junctions and the prevailing level of congestion. They are identical for all \acp{RV}, for whom the edges of $\mathcal{E}_\text{RV}$ always take $1$ time step to traverse, but they can differ between \acp{SV}, and this is modelled as demand characterized by different values of the lag index $k$. The characterization of travel times is allowed to vary with time step $t$.
\item[(A3)] The maximum number of \ac{SV} journeys starting at node $i$ and time $t$ is bounded by the number present at the end of the previous time step, $d_i^{t-1}$. In case there are insufficient shared vehicles at a node $i$ and time $t$ for the sum of all demand, i.e.~$d_i^{t-1} < \sum_{j,k}f_{i,j}^{t,k}$, the system operator can choose which journeys should take place. This contrasts with the approach of \cite{ghosh_dynamic_2017}, which can be viewed as conservative.\footnote{Ghosh \textit{et al.}~\cite{ghosh_dynamic_2017} impose an additional ``fair sharing'' constraint meaning that, in our notation, integer valued $w_{i,j}^{t,k}$ are no greater than their demand share of the $d_i^{t-1}$ available \acp{SV}. Thus, for example, if for some $(i,t,k)$ there are 9 \acp{SV} present, and 10 destinations $j$ for which $f_{i,j}^{t,k} = 1$, no journeys can occur at all, because one must enforce $w_{i,j}^{t,k} \leq \tfrac{9}{10}$ for each destination $j$.}
\end{enumerate}

\subsection{Network flow formulation of \ac{DRRP}}

For a given realization $\xi \in \Xi$, the \ac{DRRP} is presented in problem \eqref{eq:DRRP}. This can be viewed as a conversion of the \ac{DRRP} formulation of \cite[Table 4]{ghosh_dynamic_2017} into an equivalent form, in the sense that the two problems have the same optimal value and optimal solutions can be mapped from one form to the other. The system operator controls flows of \acp{RV} $z_{i,j}^t$, flows of \acp{SV} $b_{i,j}^t$, and load/unload decisions $y_i^{+,t}$ and $y_i^{-,t}$ at stations, on the assumption that customers then make optimal use of the resulting arrangement to complete their journeys. In our new formulation, the flows of \acp{SV} are ``nested'' within those of \acp{RV}, in the sense that on any edge $(i,j)$ at time $t$, \ac{SV} flows $b_{i,j}^t$ for rebalancing are constrained by the storage capacity of the \acp{RV} that are moving along the same edge, equal to $\overline{b}z_{i,j}^t$. The nesting effect is illustrated in Fig.~\ref{fig:twographs}.
\begin{subequations}\label{eq:DRRP}
\allowdisplaybreaks
\begin{align}
\min_{z,y,b,w,d} \,\, & \sum_{t=1}^T  \left[ \sum_{(i, j)\in \mathcal{E}_\text{SV}}   \sum_{k=0}^K l_{i,j}^{t,k} (f_{i,j}^{t,k}(\xi) - w_{i,j}^{t,k}; \xi) + \sum_{(i,j) \in \mathcal{E}_\text{RV}}  c_{i,j}^t z_{i,j}^t + \sum_{i \in \mathcal{N}_\text{SV} \cap \mathcal{N}_\text{RV}} r_i^t (y_i^{+,t} + y_i^{-,t})\right] \\
\text{s.~t.}\quad & d_i^t = d_i^{t-1} + \sum_{k=0}^K \left( \sum_{(j,i) \in \mathcal{E}_\text{SV}} w_{j,i}^{t-k,k} - \sum_{(i,j) \in \mathcal{E}_\text{SV}} w_{i,j}^{t,k} \right) + y_i^{-,t} - y_i^{+,t} \, , \,\, t=1,\ldots,T, \, \, \forall \, i \in \mathcal{N}_\text{SV} \, , \label{eq:DRRPNC}\\
& \sum_{(i,j)\in\mathcal{E}_\text{RV}} \!\!\! b_{i,j}^t = \sum_{(j,i)\in\mathcal{E}_\text{RV}} \!\!\! b_{j,i}^{t-1} + y_i^{+,t} - y_i^{-,t}\, , \,\, t=1,\ldots,T,\,\, \forall \, i \in \mathcal{N}_\text{RV} \, ,\label{eq:DRRPSVC}\\
& \sum_{(i.j)\in\mathcal{E}_\text{RV}} \!\!\! z_{i,j}^t = \sum_{(j,i)\in\mathcal{E}_\text{RV}} \!\!\! z_{j,i}^{t-1} \, , \,\, t = 1,\ldots,T, \,\, \forall \, i \in \mathcal{N}_\text{RV} \, , \label{eq:DRRPRVC}\\
& 0 \leq b_{i,j}^t \leq \overline{b} z_{i,j}^t , \, t=0,\ldots,T\,\,  \text{and} \,\, b_{i,j}^0 \in \mathbb{Z} \,\, \text{given,} \,\,\forall \, (i,j) \in \mathcal{E}_\text{RV} \,, \label{eq:DRRPbB} \\
& 0 \leq w_{i,j}^{t,k} \leq f_{i,j}^{t,k}(\xi) \,  , \,\, t=1,\ldots,T\,, \,\,k=0,\ldots,K \, , \nonumber \\
& \hspace{2cm} \text{and} \,\, w_{i,j}^{t,k} \,\, \text{given for} \,\,1 - K \leq t \leq 0 \, , \,\, -t < k \leq K \, , \,\forall \, (i,j) \in \mathcal{E}_\text{SV} \,, \label{eq:DRRPwB}\\
& 0 \leq d_i^t \leq \overline{d}_i\,  , \,\, t=0,\ldots,T\,\, \text{and}\,\, d_i^0 \in \mathbb{Z} \,\, \text{given,} \,\, \forall \, i \in \mathcal{N}_\text{SV} \, , \label{eq:DRRPdB}\\
& 0 \leq y_i^{+,t} \leq \overline{y} \, , \quad 0 \leq y_i^{-,t} \leq \overline{y} \, , \,\, t = 1, \ldots, T \, , \,\, \forall \, i \in \mathcal{N}_\text{SV} \cap \mathcal{N}_\text{RV} \, , \label{eq:DRRPyB}\\
& z_{i,j}^0 \,\, \text{given such that}\,\,\sum_{i,j} z_{i,j}^0 = |\mathcal{V}| \, , \,\, \forall \, (i,j) \in \mathcal{E}_\text{RV}\label{eq:DRRPzB} \\
& w_{i,j}^{t,k} \in \mathbb{Z}  \,\, \text{for all tuples $(i,j,t,k)$.} \label{eq:DRRP2Int3} \\
& y_i^{+,t}, y_i^{-,t}, b_{i,j}^t \in \mathbb{Z}  \,\, \text{for all indices.} \label{eq:DRRP2Int2} \\
& z_{i,j}^t \in \mathbb{Z} \, , \quad t = 1, \ldots, T, \,\,  \forall \, (i,j) \in \mathcal{E}_\text{RV}\, , \label{eq:DRRP2Int}
\end{align}
\end{subequations}

Constraint \eqref{eq:DRRPNC} enforces conservation of \acp{SV} at each customer node $i$ and time $t$. The model allows for \acp{SV} to ``jump'' from station $i$ to $j$ during a time step, modelling very short journeys as duration $k = 0$ discrete steps.

Constraint \eqref{eq:DRRPSVC} enforces the conservation of \ac{SV} flows on $\mathcal{G}_\text{RV}$, i.e., while the \acp{SV} are loaded inside \acp{RV}. Constraint \eqref{eq:DRRPRVC} enforces conservation of the \acp{RV} themselves, and \eqref{eq:DRRPbB} bounds the number of \acp{SV} travelling along an edge in $\mathcal{E}_\text{RV}$  by the total capacity of \acp{RV} travelling along the same edge. Constraints \eqref{eq:DRRPSVC} and \eqref{eq:DRRPbB} together prevent transport of \acp{SV} on edges $\mathcal{E}_\text{RV}$ unless an \ac{RV} is present at the right node and time.

Constraint \eqref{eq:DRRPwB} bounds the number of journeys by the level of demand present for each tuple $(i,j,t,k)$. As our formulation is intended for a real-time implementation, we allow some journeys already to be in progress at time $0$, with this having implications for constraint \eqref{eq:DRRPNC}. Constraint \eqref{eq:DRRPdB} bounds the number of \acp{SV} that can be accommodated at each node $i \in \mathcal{N}_\text{SV}$, and specifies the initial ``fill level'' at each node. Constraint \eqref{eq:DRRPyB} limits loading and unloading actions to be nonnegative, and upper-bounds them by some constant $\overline{y} \leq \overline{b}|\mathcal{V}|$ representing the largest single rebalancing action contemplated by the operator.

Constraint \eqref{eq:DRRPzB} specifies that \ac{RV} flows must correspond to integer-valued numbers of \acp{RV}, with initial conditions corresponding to the number of \acp{RV} $|\mathcal{V}|$ present. Constraints \eqref{eq:DRRP2Int} and \eqref{eq:DRRP2Int2} specify that \ac{SV} transport flows, \ac{RV} movement decisions, \ac{SV} loading/unloading actions, and customer journeys are all integer-valued.

For brevity we use symbols $z,y,b,w,d$ to refer to concatenations of their corresponding indexed quantities.\\

\begin{rem}
In addition to the social cost of unserved demand, $\sum_{i,j,t,k}l_{i,j}^{t,k}(w_{i,j}^{t,k})$, one may also be interested in the \textit{service rate}, defined as the ratio of completed journeys to total demand, $\sum_{i,j,t,k}w_{i,j}^{t,k} / \sum_{i,j,t,k}f_{i,j}^{t,k}$. This metric has the convenience of being invariant to system size. The service rate can be accommodated as an objective function by setting $l_{i,j}^{t,k}(x; \xi) = x$ and all other costs to zero. The objective function to minimize, when normalized by $\sum_{i,j,t,k} f_{i,j}^{t,k}(\xi)$, will then be equal to $1$ minus the service rate.
\end{rem}

Finally, we prove two elementary results. The first is that although the new formulation allows $y_i^{+,t} > 0$ and $y_i^{-,t} > 0$ simultaneously in feasible solutions, optimal solutions can always be generated that do not feature this. The second is that there always exists an optimal solution for which $y_i^{+,t} = y_i^{-,t} = 0$ if no \acp{RV} are present at node and time $(i,t)$. Thus optimal solutions that obey the physical constraints of \acp{RV} transporting \acp{SV} can always be generated, assuming of course the problem is feasible in the first place.\\

\begin{lem} \label{lem:minimal_y_orig}
There always exists an optimizer of \eqref{eq:DRRP} with $y$ variables taking the form $(y_i^{+,t},0)$ or $(0,y_i^{-,t})$ for all tuples $(i,t)$. If $r_i^t > 0$ for all tuples $(i,t)$, then no optimal solution has $y_i^{+,t} > 0$ and $y_i^{-,t} > 0$ simultaneously for any $(i,t)$.
\end{lem}
\begin{proof} By inspecting the constraints and recalling that the cost coefficients satisfy $r_i^t \geq 0$, it is immediate that for any elements $(y_i^{+,t},y_i^{-,t})$ of a feasible solution to \eqref{eq:stage1_approx}, the modification $(y_i^{+,t} - \min\{y_i^{+,t},y_i^{-,t}\},$ $y_i^{-,t} - \min\{y_i^{+,t},y_i^{-,t}\})$ is feasible and has no greater cost. The second claim follows from the fact that the modification results in strictly lower cost.
\end{proof}

\begin{lem} \label{lem:y_only_with_rv}
If no \ac{RV} passes through node $i$ at time $t$, then for any instance of problem \eqref{eq:DRRP} there is an optimal solution such that $y_i^{+,t} = y_i^{-,t} = 0$, i.e.~that load/unload actions only take place where there are \acp{RV}. If $r_i^t > 0$ for all tuples $(i,t)$, then no optimal solution has $y_i^{+,t} = y_i^{-,t} > 0$ for any $(i,t)$ where no \ac{RV} is present.
\end{lem}
\begin{proof} If no \ac{RV} passes through node $i$ at time $t$, then constraint \eqref{eq:DRRPSVC} implies $y_i^{+,t} = y_i^{-,t}$. If for some tuple $(i,t)$ $y_i^{+,t} = y_i^{-,t} > 0$ in an optimal solution, then both can be reduced to zero without affecting feasibility. The second claim follows from the fact that reducing both to zero strictly lowers the cost, thus the original solution could not have been optimal.
\end{proof}

\subsubsection{Strength of relaxation} \label{sec:rel_strength}

Numerical tests show the \ac{LP} relaxation of problem \eqref{eq:DRRP} to be at least as tight, and very often much tighter, than that of the mixed-integer formulation presented in \cite[Table 4]{ghosh_dynamic_2017}. A simple example provides intuition as to why this should be the case. We note that the formulation in \cite{ghosh_dynamic_2017} uses only a single variable $d_v^{*,t}$ to represent the number of \acp{SV} inside a \ac{RV} $v$ at time $t$, but multiple variables $y_{i,v}^{+,t}$, $y_{i,v}^{-,t}$ linked by a single constraint \cite[constraint (4)]{ghosh_dynamic_2017} to update $d_v^{*,t}$ after load/unload actions anywhere on the network. 

Consider a 3-node bike-sharing system shown in Fig.~\ref{fig:3node}, in which a single truck able to carry at least 2 bikes starts at node 2 at $t=0$, and wants to bring a bike from node 1 to node 3 where it is needed to satisfy a unit of demand at time $t=2$. If respecting real-world constraints, it clearly takes 3 time steps to for the truck to reach node 1 and then bring the bike to node 3. In the LP relaxation of \cite[Table 4]{ghosh_dynamic_2017}, the truck splits into two, $z_{2,1}^1 = z_{2,3}^1 = \tfrac{1}{2}$, and loads the bike onto the ``half-truck'' at node 1 at the next time step. It can simultaneously drop the bike at node 3, because in that formulation the locations of bikes \emph{loaded on trucks} are not represented, and the move is consistent with \cite[constraint (4)]{ghosh_dynamic_2017}. Thus the bike in effect jumps from node 1 to node 3, and the demand there will be served. The move is only feasible because there is ``half a truck'' at nodes 1 and 3 at the same time. In contrast, in the LP relaxation of problem \eqref{eq:DRRP} the bike itself must take 2 time steps to travel from nodes 1 to 3, as this would otherwise violate the constraints on \ac{SV} flow variables $b_{i,j}^t$. Thus the bike cannot be provided at node 3 in time to service the demand. In short, formulation \eqref{eq:DRRP}, despite returning the same objective when integrality constraints are enforced, prevents the ``teleportation'' of \acp{SV} when these are relaxed, and its \acp{LP} relaxation therefore does not overstate the abilities of the \acp{RV} as much.

\begin{figure}[tbp]
\begin{center}
\includegraphics[width=0.6\textwidth]{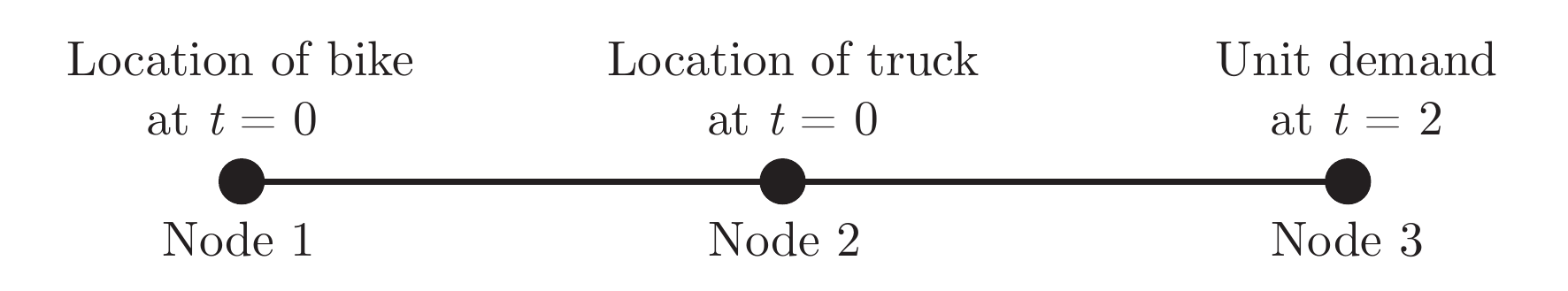}
\caption{3-node network used in Section \ref{sec:rel_strength} to illustrate the difference between LP relaxations.}
\label{fig:3node}
\end{center}
\end{figure}

We verified the intuition by testing the strength of the relaxation on deterministic benchmark instances of a bike-sharing problem, which were created via the same procedure and parameters described in Section \ref{sec:test_networks}, and solved on the same hardware, as for the main numerical results of this paper. For a fair comparison with \cite{ghosh_dynamic_2017}, all loss functions were modelled as deterministic linear functions $l_{i,j}^{t,k}(x) = x$, and the ``fair sharing'' constraint \cite[eq.~(3)]{ghosh_dynamic_2017} was not enforced.\footnote{Qualitatively similar results were obtained when this constraint was added to both models, but we omit these for brevity.} Table \ref{tab:relax} shows average results over 10 instances for each row, using the standardized test networks described in Section \ref{sec:test_networks}. Problems were solved in Gurobi to within a relative tolerance of $10^{-3}$, and a time limit of 600 seconds was enforced. The \textit{LP gap} columns report the average percentage decrease in solver objective when integrality constraints were relaxed, relative to the best solution of \eqref{eq:DRRP} found in the time limit. In the \textit{Solution time} columns, the \textit{LP} columns indicate time for the relaxations, and the others indicate MIP solve times. Perhaps as a side effect of the tighter \ac{LP} relaxation the full MIP in our formulation was also solved considerably faster on average, except for very simple problems with only one \ac{RV}. The final two columns indicate respectively the average service rates with no operator intervention and with interventions returned by solving \eqref{eq:DRRP} until timeout or an optimal solution was found. Test results are shown for sizes of system that could be solved directly as an MIP across multiple instances within a reasonable time.

\begin{table}[tbp] 
\caption{Comparison of LP relaxation strengths of formulations \eqref{eq:DRRP} and \cite[Table 4]{ghosh_dynamic_2017}. Results are averaged over 10 random instances. The better (lower) value between the two formulations is highlighted in bold.}
\begin{center}
{\footnotesize
\begin{tabular}{cc|rr|rr|rr|r|rr}
\toprule
   &   &   \multicolumn{2}{|c|}{Rel.~LP gap (\%)}  & \multicolumn{4}{|c|}{Average solution time (seconds)} &  Expected &  \multicolumn{2}{|c}{Service rate (\%)} \\
$|\mathcal{N}_\text{SV}|$ & $|\mathcal{V}|$ &  \hspace{0.5cm}\eqref{eq:DRRP} &  \cite{ghosh_dynamic_2017} & \eqref{eq:DRRP} LP & \eqref{eq:DRRP} &  \cite{ghosh_dynamic_2017} LP & \cite{ghosh_dynamic_2017} &  demand &  No action  & \eqref{eq:DRRP} \\
\midrule
4  & 1 &  26.8      & 26.8 & 0.003 &    0.02 &  \bt{0.002} & \bt{0.01} &        31.4 &          93.2 &         98.9 \\
9  & 1 &  \bt{25.7} & 34.9 & 0.015 & \bt{0.32} &  \bt{0.007} &    0.36 &        76.2 &          85.9 &         98.1 \\
   & 3 &  \bt{8.5}  & 21.5 & 0.016 & \bt{0.28} &      0.016 &   21.11 &           ''  &             ''  &         99.5 \\
16 & 1 &  \bt{51.4} & 58.8 & 0.060 & 3.44      &  \bt{0.024} & \bt{2.23} &       151.8 &          84.3 &         93.8 \\
   & 3 &  \bt{3.8}  & 10.3 & \bt{0.054} &   \bt{88.76} &  0.058 &  249.06 &          ''   &           ''    &         98.3 \\
25 & 1 &  \bt{30.7} & 57.4 &     0.135 &  \bt{10.61} &  \bt{0.072} &   15.09 &       232.2 &          80.9 &         88.4 \\
   & 3 &  \bt{27.4} & 29.9 & 0.138 &  \bt{194.64} & \bt{0.119} &  356.16 &          ''   &         ''      &         95.4 \\
   & 5 &  \bt{15.1} & 18.0 & \bt{0.144} &  \bt{213.47} &      0.205 &  493.22 &      ''       &      ''         &         97.8 \\
36 & 1 &  \bt{19.2} & 39.6 & 0.233 &   12.86 &  \bt{0.126} & \bt{7.60} &       321.5 &          81.5 &         86.9 \\
   & 3 &  \bt{35.8} & 49.0 & 0.298 &  \bt{382.59} &  \bt{0.247} &  398.18 &          ''   &          ''     &         93.7 \\
   & 5 &  \bt{24.1} & 32.0 & \bt{0.292} &  \bt{182.90} &      0.388 &  531.38 &       ''      &       ''        &         96.6 \\
   & 7 &  \bt{7.6}  & 20.8 & \bt{0.290} &  \bt{154.87} &      0.581 &  528.69 &        ''     &       ''        &         97.1 \\
\bottomrule
\end{tabular}
}
\end{center}

\label{tab:relax}
\end{table}

%

\subsection{Two-stage stochastic program}

The operator wishes to solve \eqref{eq:DRRP} to minimize expected costs over all possible realizations of $\xi$. We consider a setting in which the operator chooses decisions $z$ and $y$ before observing $\xi$, and for the purposes of optimization assumes that $z$ and $y$ are fixed decisions over the planning horizon. In stochastic programming terminology, the decision-maker optimizes without recourse. In an industrial implementation, we intend that our formulation would be used in a rolling look-ahead context, in which only the initial step, or steps, of $z$ and $y$ are carried out before a re-optimization takes place. The new optimization problem will be re-parameterized using the new system state and elements of $\xi$ observed since the last one was carried out. Thus, although there are no explicit recourse variables in our formulation, a form of recourse arises from the re-optimizations that are carried out in real time.

The decision problem can be cast as a two-stage stochastic program. The first stage is
\begin{subequations} \label{eq:stage1}
\allowdisplaybreaks
\begin{align}
    \min_{z,y,b} \quad & \sum_{t=1}^T \! \left[ \sum_{(i,j) \in \mathcal{E}_\text{RV}} \!\!\! c_{i,j}^t z_{i,j}^t + \sum_{i \in \mathcal{N}_\text{SV} \cap \mathcal{N}_\text{RV}} r_i^t (y_i^{+,t} + y_i^{-,t}) \right] + \mathbb{E}[V(y, \xi)] \label{eq:stage1_obj} \\
    \text{s.~t.}\quad & \sum_{(i,j)\in\mathcal{E}_\text{RV}} \!\!\! b_{i,j}^t = \sum_{(j,i)\in\mathcal{E}_\text{RV}} \!\!\! b_{j,i}^{t-1} + y_i^{+,t} - y_i^{-,t}\, , \,\, t=1,\ldots,T,\,\, i \in \mathcal{N}_\text{RV} \, ,\label{eq:DRRPSVC2}\\
& \sum_{(i.j)\in\mathcal{E}_\text{RV}} \!\!\! z_{i,j}^t = \sum_{(j,i)\in\mathcal{E}_\text{RV}} \!\!\! z_{j,i}^{t-1} \, , \,\, t = 1,\ldots,T, \,i \in \mathcal{N}_\text{RV} \, , \label{eq:DRRPRVC2}\\
& 0 \leq b_{i,j}^t \leq \overline{b} z_{i,j}^t , \, t=0,\ldots,T\,\, \text{and} \,\, b_{i,j}^0 \in \mathbb{Z} \,\, \text{given,} \label{eq:DRRPbB2} \\
& 0 \leq y_i^{+,t} \leq \overline{y} \, , \quad 0 \leq y_i^{-,t} \leq \overline{y} \, , \,\, t = 1, \ldots, T \, , \,\, i \in \mathcal{N}_\text{SV} \, , \label{eq:DRRPyB2}\\
& z_{i,j}^0 \in \mathbb{Z} \,\, \text{given such that}\,\,\sum_{i,j} z_{i,j}^0 = |\mathcal{V}| \, , \label{eq:DRRPzB2} \\
& y_i^{+,t}, y_i^{-,t}, b_{i,j}^t \in \mathbb{Z}  \,\, \text{for all indices.} \label{eq:DRRP2Int2a} \\
& z_{i,j}^t \in \mathbb{Z} \, , \quad t = 1, \ldots, T, \,\,  (i,j) \in \mathcal{E}_\text{RV}\, , \label{eq:DRRP2Inta}
\end{align}
\end{subequations}
in which $V(y,\xi)$ represents the second-stage cost resulting from rebalancing actions $y$, for demand realization $\xi$. The second-stage problem, in which $\xi$ appears as a fixed parameter, is
\begin{subequations} \label{eq:stage2}
\allowdisplaybreaks
\begin{align}
    V(y, \xi) :=  \min_{w} \quad & \sum_{t=1}^T \sum_{(i, j)\in \mathcal{E}_\text{SV}} \sum_{k=0}^K l_{i,j}^{t,k} (f_{i,j}^{t,k}(\xi) - w_{i,j}^{t,k}; \xi) \label{eq:DRRPS2Obj}\\
    \text{s.~t.}\quad & 0 \leq d_i^0 + \sum_{\tau=1}^t\left[ \sum_{k=0}^K \left( \sum_{(j,i) \in \mathcal{E}_\text{SV}} \!\!\! w_{j,i}^{\tau-k,k} - \!\!\! \sum_{(i,j) \in \mathcal{E}_\text{SV}} \!\!\!\!\! w_{i,j}^{\tau,k} \right) \!\! + y_i^{-,\tau} \! - \! y_i^{+,\tau} \right] \leq \overline{d}_i \,,  t=1,\ldots,T, \, i \in \mathcal{N}_\text{SV}, \label{eq:DRRPNC2}\\
    & 0 \leq w_{i,j}^{t,k} \leq f_{i,j}^{t,k}(\xi) \,  , \,\, t=1,\ldots,T\,, \,\,k=0,\ldots,K \, , \nonumber \\
& \hspace{2cm} \text{and} \,\, w_{i,j}^{t,k} \,\, \text{given for} \,\,1 - K \leq t \leq 0 \, , \,\, -t < k \leq K \, , \,\forall \, (i,j) \in \mathcal{E}_\text{SV} \,, \label{eq:DRRPwB2}\\
& w_{i,j}^{t,k} \in \mathbb{Z}  \,\, \text{for all tuples $(i,j,t,k)$.} \label{eq:DRRP2Int3a}
\end{align}
\end{subequations}

Constraint \eqref{eq:DRRPNC2} combines \eqref{eq:DRRPSVC} and \eqref{eq:DRRPdB}. This eliminates the station fill level variables $d_i^t$, which are used in \eqref{eq:DRRP} only for exposition, by writing $d_i^t  = d_i^0 + \sum_{\tau=1}^t\left[ \sum_{k=0}^K \left( \sum_{(j,i) \in \mathcal{E}_\text{SV}} \!\! w_{j,i}^{\tau-k,k} - \sum_{(i,j) \in \mathcal{E}_\text{SV}} \!\! w_{i,j}^{\tau,k} \right) + y_i^{-,\tau} - y_i^{+,\tau} \right]$.

\subsection{Second stage \ac{LP} relaxation and feasibility}

We now prove an important, advantageous property of the second stage in our formulation, and then discuss the feasibility of problem \eqref{eq:stage2}.\\

\begin{prop}[Tight LP relaxation] \label{prop:stage2int}
If the constraints of problem \eqref{eq:stage2} have integer-valued right-hand sides, the problem has an integer-valued solution even when when integrality constraint \eqref{eq:DRRP2Int3a} is relaxed.
\end{prop}
\begin{proof}
We show that the optimization over $w$ can be written as a min-cost flow problem \cite[\S 5]{ahuja_network_1993} over a directed graph with integer-valued sources, sinks, and edge capacities. It is widely known \cite[\S 5.5]{powell_stochastic_2003} that this guarantees the existence of an integer-valued \ac{LP} solution.

We first note that since $l_{i,j}^{t,k}(\, \cdot \, , \xi)$ is convex and piecewise linear,  with breakpoints at every integer argument $\{1, \ldots, f_{i,j}^{t,k}(\xi)\}$, it can be represented by the cost of the flow shown in Fig.~\ref{fig:s2flowa}, in which each edge supports a unit of demand, with cost equal to the negated gradient of the relevant segment of $l_{i,j}^{t,k}(\, \cdot \, , \xi)$. In any optimal solution to the flow problem, the edges with the most negative cost, corresponding to the most valuable journeys and the steepest segments of $l_{i,j}^{t,k}(\, \cdot \, , \xi)$, will be used first. More precisely, the edge flows model $l_{i,j}^{t,k}(f_{i,j}^{t,k}(\xi) - w_{i,j}^{t,k} , \xi) - l_{i,j}^{t,k}(f_{i,j}^{t,k}(\xi),\xi)$, the latter term being the cost of leaving all customers unserved.

Second, the fill levels of stations can be represented as shown in Fig.~\ref{fig:s2flowb}. The flow on each horizontal link equals $d_i^t$ and is constrained to the interval $[0, \overline{d}_i]$. At each time step it receives contributions from arriving and departing \acp{SV}, each modelled as described in the previous paragraph, and is subject to load/unload actions $y_i^{+,t}$ and $y_i^{-,t}$. To avoid overcrowding only $y_i^{+,1}$ and $y_i^{-,1}$ are labelled in the figure.

\ac{SV} flows already in progress, $w_{i,j}^{t,k}$ for $1-K \leq t \leq 0$, whose values are fixed data, are modelled as fixed supply nodes connected to the relevant receiving node of Fig.~\ref{fig:s2flowb}. \ac{SV} flows for which $t+k > T$ are modelled as in Fig.~\ref{fig:s2flowa}, but their end nodes are connected to the sink. To avoid overcomplication, neither of these two types of flow are shown in Fig.~\ref{fig:s2flowb}.

Lastly, the flows exiting each station row are routed to a common sink with demand set to conserve \acp{SV} in the system. Thus the problem can be written in canonical form $\min_{x} c^\top x$ subject to $\sum_{m \rightarrow n} x_{m} - \sum_{n \rightarrow p} x_{p} = b_n\,\, \forall n$ and $0 \leq x \leq \overline{x}$, where the scalars $b_n$ and the elements of $\overline{x}$ are all integer valued. See Appendix A for a full description.
\end{proof}

\begin{figure}[t]
\centering
\begin{subfigure}[t]{0.38\textwidth}
\centering
\begin{tikzpicture}[scale=1]
	\coordinate (i) at (0.5, 4); \coordinate (j) at (2.5, 0);
	
	\draw[myPath] (i) to [out=260, in=125] (0.4, 1.2) to [out=305, in=160] (j);
	\draw[myPath] (i) to [out=270, in=120] (j);
	\draw[-, dotted, thick] (1.4, 2) -- (1.8, 2.2);
	\draw[myPath] (i) to [out=310, in=80] (j);
	\draw node at (-0.15, 1.1) {$\overline{l}_{i,j}^{t,k,1}$};
	\draw node at (0.9, 1.5) {$\overline{l}_{i,j}^{t,k,2}$};
	\draw node at (2.6, 2.5) {$\overline{l}_{i,j}^{t,k,f_{i,j}^{t,k}}$};
	\filldraw[gray] (i) circle (2pt) node[left] {$(i, t)$};
	\filldraw[gray] (j) circle (2pt) node[below] {$(j, t+k)$};
\end{tikzpicture}
\vspace{0.55cm}  
\caption{Model of flow $w_{i,j}^{t,k}$ as the sum of flows along unit-capacity edges, whose costs $\overline{l}_{i,j}^{t,k,\cdot}$ are negated slopes of the piecewise linear cost function $l_{i,j}^{t,k}(f_{i,j}^{t,k}(\xi) - w_{i,j}^{t,k},\xi)$. Optimal flow solutions are equivalent to objective \eqref{eq:DRRPS2Obj} when the constant $l_{i,j}^{t,k}(f_{i,j}^{t,k}(\xi);\xi)$ is added.}
\label{fig:s2flowa}
\end{subfigure}\hspace{1em}
\begin{subfigure}[t]{0.57\textwidth}
\begin{tikzpicture}
	\coordinate (i1) at (0,3); \coordinate (i2) at (2,3);
	\coordinate (iT1) at (4,3); \coordinate (iT) at (6,3); \coordinate (iTend) at (8,3);
	\coordinate (j1) at (0,0); \coordinate (j2) at (2,0);
	\coordinate (jT1) at (4,0); \coordinate (jT) at (6,0); \coordinate (jTend) at (8,0);
	\draw[myPath] (i1) -- (i2) node [midway, above] {$d_i^1$};
	\draw[dashed] (i2) -- (iT1) node [midway, above] {$\cdots$};
	\draw[myPath] (iT1) -- (iT) node [midway, above] {$d_i^{T-1}$}; 
	\draw[myPath] (iT) -- (iTend) node [midway, above] {$d_i^T$};

	\node[red] (di0) [above of=i1, yshift=0.1cm] {$d_i^0$};
	\draw[myPath2] (di0) -- (i1);
	\draw[myPath2] (i1){}+(-0.5,-0.5) -- (i1);
	\draw[myPath2] (i1) -- (0,2.3);
	\draw[red] (i1){}+(-0.6,-0.7) node {$y_i^{-,1}$};
	\draw[red] (i1){}+(0,-0.9) node {$y_i^{+,1}$};
	\draw[myPath2] (i2){}+(-0.5,-0.5) -- (i2);
	\draw[myPath2] (i2) -- (2,2.3);
	\draw[myPath2] (iT1){}+(-0.5,-0.5) -- (iT1);
	\draw[myPath2] (iT1) -- (4,2.3);
	\draw[myPath2] (iT){}+(-0.5,-0.5) -- (iT);
	\draw[myPath2] (iT) -- (6,2.3);
	
	\draw[myPath] (j1) -- (j2) node [midway, above] {$d_j^1$};
	\draw[dashed] (j2) -- (jT1) node [midway, above] {$\cdots$};
	\draw[myPath] (jT1) -- (jT)node [midway, above] {$d_j^{T-1}$};
	\draw[myPath] (jT) -- (jTend) node [midway, above] {$d_j^T$};

	\node[red] (dj0) [above of=j1, yshift=0.1cm] {$d_j^0$};
	\draw[myPath2] (dj0) -- (j1);
	\draw[myPath2] (j1){}+(-0.5,-0.5) -- (j1);
	\draw[myPath2] (j1) -- (0,-0.7);
	\draw[myPath2] (j2){}+(-0.5,-0.5) -- (j2);
	\draw[myPath2] (j2) -- (2,-0.7);
	\draw[myPath2] (jT1){}+(-0.5,-0.5) -- (jT1);
	\draw[myPath2] (jT1) -- (4,-0.7);
	\draw[myPath2] (jT){}+(-0.5,-0.5) -- (jT);
	\draw[myPath2] (jT) -- (6,-0.7);
	\draw[myPath] (i1) to [out=290, in=100] (j2);
	\draw (j2){}+(-0.6,1.7) node {$w_{i,j}^{1,1}$};
	\draw[red] (jTend){}+(0.1,1.5) node {sink};
	\draw[myPath2] (iTend) to [out=330, in=100] (8.2,1.7);
	\draw[myPath2] (jTend) to [out=30, in=260] (8.2,1.3);
	\filldraw[gray] (i1) circle (2pt) node[above left] {$(i, 1)$};
	\filldraw[gray] (i2) circle (2pt) node[below right] {$(i, 2)$};
	\filldraw[gray] (iT1) circle (2pt) node[below right] {$(i, T\!-\!1)$};
	\filldraw[gray] (iT) circle (2pt) node[below right] {$(i, T)$};
	\filldraw[gray] (j1) circle (2pt) node[below right] {$(j, 1)$};
	\filldraw[gray] (j2) circle (2pt) node[below right] {$(j, 2)$};
	\filldraw[gray] (jT1) circle (2pt) node[below right] {$(j, T\!-\!1)$};
	\filldraw[gray] (jT) circle (2pt) node[below right] {$(j, T)$};
\end{tikzpicture}
\caption{Schematic of second stage as a min-cost flow problem. For clarity only two station rows $i,j$ and one \ac{SV} flow $w_{i,j}^{1,1}$ are shown. Each flow $w_{i,j}^{t,k}$ is modelled as illustrated in panel (a). The sink value is equal to the sum of net supply elsewhere in the system. Nodes and their labels are depicted in grey. Each horizontal edge has cost zero and capacity $\overline{d}_i$ or $\overline{d}_j$.}
\label{fig:s2flowb}
\end{subfigure}
\caption{Model of the second stage \eqref{eq:stage2} as a min-cost flow problem.}
\label{fig:s2flow}
\end{figure}

\begin{rem}
Problem \eqref{eq:stage2} can be solved using a dedicated min-cost flow solver, which can be more efficient than a solver designed for generic \acp{LP}. This solver must be capable of returning the Lagrange multipliers required in line \ref{algl:stage2} of Algorithm \ref{alg:spar}. Another way of obtaining these multipliers would be to solve the corresponding dual problem \cite[eq.~(5.2)]{ahuja_network_1993}.
\end{rem}


\subsubsection{Feasibility and model complexity} \label{sec:feas} \label{sec:AddAss1}

Problem \eqref{eq:stage2} may not be feasible, depending on the boundary conditions $y_i^{+,t},y_i^{-,t}$ inherited from the first stage. As a trivial example, consider a system with zero demand ($f_{i,j}^{t,k}=0$ for all $i,j,t,k$), in which one station $i$ has an initial fill level satisfying $0 \leq d_i^0 \leq \overline{d}_i$ and immediately receives $y_i^{-,1} > \overline{d}_i - d_i^0$ \acp{SV} in the first time step, with $y_i^{+,t} = 0$. With all $w_{i,j}^{t,k}$ constrained to zero by $f_{i,j}^{t,k}$, constraint \eqref{eq:DRRPNC2} cannot be satisfied.

However, reality is more complex than our model, in particular regarding the no-recourse assumption on $y$ and $z$ and Assumption (A1) concerning customer behaviour. In practice, a system operator finding the load/unload schedules $y_i^{+,t}$ and $y_i^{-,t}$ infeasible when trying to implement them would simply make a best effort at the time, possibly ordering additional \ac{RV} trips in the process. On the demand side, a customer would overcome the problem of a full end station $j$ by diverting to another station $j'$ on arrival, and an empty initial station $i$ may lead to the customer using a nearby station $i'$ instead. Thus true infeasibility is never encountered in the real world. Conceptually, it would be possible to address some of these complexities in our model by introducing a large number of additional second-stage decision variables for the extra ``hops'' $(i, i')$, $(i', j)$, and $(j, j')$ made by inconvenienced customers. However, modelling the costs and causal customer behaviour associated with these additional variables would be a substantial endeavour. Additional complexities exist beyond the feasibility issue; for example, a risk-averse customer seeing the current state of the system may decline to attempt any journey at all, even if there is in fact a reasonable chance of completing it successfully.

As a compromise, we use a modelling simplification to ensure solutions respect constraint \eqref{eq:DRRPNC2} whenever possible, and otherwise neglect the above issues. We introduce two extra penalty variables per node $i \in \mathcal{N}_\text{SV}$ and time step $t$, $p_i^{+,t}$ and $p_i^{-,t}$, which allow \acp{SV} to be created or destroyed at a high cost for modelling purposes. The variables appear in the objective function of the second stage with a suitably high linear cost coefficient $r_p$. We therefore solve the following modified version of problem \eqref{eq:stage2}:
\begin{subequations} \label{eq:stage2Pen}
\allowdisplaybreaks
\begin{align}
\min_{w} \quad & \sum_{t=1}^T \left[  \sum_{(i, j)\in \mathcal{E}_\text{SV}} \sum_{k=0}^K l_{i,j}^{t,k} (f_{i,j}^{t,k}(\xi) - w_{i,j}^{t,k}; \xi) + \sum_{i \in \mathcal{N}_\text{SV}} r_p (p_i^{+,t} + p_i^{-,t})\right] \label{eq:DRRPS2ObjPen}\\
    \text{s.~t.}\quad & 0 \leq d_i^0 + \sum_{\tau=1}^t\left[ \sum_{k=0}^K \left( \sum_{(j,i) \in \mathcal{E}_\text{SV}}  w_{j,i}^{\tau-k,k} - \sum_{(i,j) \in \mathcal{E}_\text{SV}}  w_{i,j}^{\tau,k} \right) + y_i^{-,\tau} - y_i^{+,\tau} \right] \nonumber \\
    & \hspace{5cm} + p_i^{+,t} - p_i^{-,t}\leq \overline{d}_i \,, \quad  t=1,\ldots,T, \, i \in \mathcal{N}_\text{SV}, \label{eq:DRRPNC2Pen}\\
        & 0 \leq p_i^{+,t} \, , \quad 0 \leq p_i^{-,t} \, , \quad t=1,\ldots,T, \, i \in \mathcal{N}_\text{SV}, \label{eq:Pgt0}\\
& p_i^{+,t}, p_i^{-,t} \in \mathbb{Z}\, ,  \quad t=1,\ldots,T, \, i \in \mathcal{N}_\text{SV}, \label{eq:Pint} \\
&     \text{\eqref{eq:DRRPwB2} and \eqref{eq:DRRP2Int3a} hold.} \nonumber
\end{align}
\end{subequations}
The addition of variables $p_i^{+,t}$ and $p_i^{-,t}$ means that constraint \eqref{eq:DRRPNC2Pen} can always be satisfied, as the extra variables allow any positive or negative integer to be added to the term between the inequalities, for each node $i \in \mathcal{N}_\text{SV}$ and time step $t$. It is also straightforward to show that Proposition \ref{prop:stage2int} still holds in this setting.\footnote{The proposition is now read ``...even when integrality constraints \eqref{eq:DRRPwB2} and \eqref{eq:Pint} are relaxed.'' It can be shown to hold by modifying the graph in Fig.~\ref{fig:s2flowb} to include an additional source and sink, high-cost edges to and from each node, and zero-cost edges allowing these to be bypassed if the problem is already feasible. For the sake of brevity we omit a full description.} As the extra penalties distort the objective function for the subset of $\xi$ realizations where an infeasibility would otherwise occur, the penalty function in effect causes the operator to solve a slightly modified decision problem.

\subsection{Model calibration} \label{sec:Calibration}

The model relies on a number of parameters, as listed in Table \ref{tab:notation}. For the decision problem to generate decisions that cause a real-world improvement to welfare, these must be calibrated based on the system operator's knowledge of the system. In general, the system operator's own costs and constraints, for example the \ac{RV} movement cost coefficients $c_{i,j}^t$ and load/unload costs $r_i^t$, are easier to estimate than those on the demand side, $f_{i,j}^{t,k}$ and $l_{i,j}^{t,k}(\cdot)$. The former can be estimated from records of previous operating expenses, and assigning ``amortized'' time and fuel cost components to the model coefficients as appropriate.

Estimating customer demand and cost functions for shared mobility system is a topic of numerous dedicated studies. When estimating the mean level of demand expected in future time intervals, a simple approach is to bin records of historical \ac{SV} journeys completed according to the $(i,j,t,k)$ tuple into which they fall, and normalize across all records as appropriate. While this is a straightforward computation, it does not account for those customers who would have taken a journey had there been a \ac{SV} been available, but which could not. To identify the true underlying demand, several authors have created regression models to predict short-term demand in bike sharing systems as a function of current and recent state and exogenous variables; Giot and Cherrier reviewed several of these \cite{giot_predicting_2014}.

Beyond the mean, the \emph{distribution} of customer demand per time interval must also be estimated for the purpose of scenario generation in the model. One way of doing this is simply to assume that the historical average demand observed for a given tuple $(i,j,t,k)$ represents the rate of a Poisson arrival process, which then uniquely parameterizes the probability distribution of demand for that tuple. An alternative method is to bin the historical numbers of demand events observed per time interval for that tuple, and use the resulting histogram as the probability density function for $f_{i,j}^{t,k}$. The latter approach, while able to model non-Poisson arrival processes, requires a sufficiently large number of historical samples per tuple $(i,j,t,k)$ to build plausible estimates of the demand distributions.

The above approaches model demand in an isolated manner between tuples for simplicity; an additional feature of interest for correct representation of scenarios is the \emph{correlation} of demand between tuples. Part of the variation may be linked to latent variables common to the whole system, for example local weather. Authors such as Singhvi \textit{et al.}~\cite{singhvi_predicting_2015} have created neighbourhood regression models of these effects. Often though, underlying variables cannot be identified to explain all the variation observed in historical data without some degree of overfitting.

Lastly, the value of customer travel, encoded in the loss functions $l_{i,j}^{t,k}(\cdot)$, must also be estimated. This function can be viewed as the sum of the journey price paid plus whatever dollar-equivalent surplus the customer enjoys above this. In general, the values of historical customer journeys are known to be bounded from below by the prices they paid, as rational customers would otherwise not have chosen to use the service. For an upper bound, if one assumes that the customer chooses to use the service based on a trade-off between journey time and cost, then one can also upper-bound the journey value by the cost of an on-demand taxi service from $i$ to $j$, assuming that this is at least as fast, and at least as expensive, as using the \ac{SV}. Beyond this, loss functions are typically estimated by econometric arguments or empirical studies assessing the value of customers' time \cite{lisco_value_1968}. If different prices have been offered to customers in the past, additional clues could be gained from customers' price elasticity. An alternative setting for our proposed cost function could be one where the system operator simply maximizes revenues. In this case $l_{i,j}^{t,k}(\cdot)$ will just represent lost \ac{SV} rental fee income from customers who were unable to use the system.

\section{Solution approach} \label{sec:SA}

In the common case that the number of possible demand realizations is very large or infinite, problem \eqref{eq:stage1}-\eqref{eq:stage2} cannot be solved directly, as this would require an analytical representation of $\mathbb{E}[V(y, \xi)]$, which is generally unavailable. The number of possible realizations of $\xi$ could be modelled as exponential in the number of customer nodes where demand arises, $\mathcal{O}(\overline{f}^{|\mathcal{N}_\text{SV}|})$ where $\overline{f}$ is the number of different integer demand levels per station. Alternatively it could be infinite, for example in the case of a pure Poisson arrival process with no theoretical upper bound on demand per time interval.

One potential compromise is to write a problem resembling \eqref{eq:DRRP}, but which simultaneously encodes a limited number of realizations $(\xi_{(1)}, \xi_{(2)}, \ldots)$ sampled from the support set $\Xi$, and includes duplicate decision variables $(w_{(1)}, w_{(2)},\ldots)$ and duplicates of constraints \eqref{eq:DRRPNC2}-\eqref{eq:DRRP2Int3a} for each. The second-stage cost would be represented by the sample average of objective \eqref{eq:DRRPS2Obj}. However, the resulting scenario program has a very large number of decision variables and the number of constraints scales with the number of scenarios, as found in e.g.~the car-sharing problem in \cite{fan_carsharing_2008}. A solution by branch-and-bound is impractical beyond very small networks.

Another possibility is to use a decomposition scheme. A well-known approach for such two-stage programs is the L-shaped, or Benders, decomposition \cite{benders_partitioning_1962}, in which supporting hyperplanes of a value function for the second stage are built up in an algorithm that alternates between the first and second stages. However, the large, or indeed infinite, number of possible second-stage uncertainty realizations means that the generation of valid cuts quickly becomes intractable as the problem size grows. Alternatively, one could use a scheme such as \ac{LR} to break the problem into smaller subproblems linked by pricing functions. Such an approach was tried in \cite{ghosh_dynamic_2017}, and while it was possible to use additional clustering heuristics to solve large-scale problems, computation times remained long (significant fractions of a day), even in a deterministic demand setting.

\subsection{Separable value function approximation}

As an alternative to the decompositions described above, we propose to solve \eqref{eq:stage1}-\eqref{eq:stage2} using an approximate representation of the second-stage value function, \[\overline{V}(y; \theta) \approx \mathbb{E}[V(y,\xi)]\] taking the form
\begin{equation} \label{eq:vbardef}
    \overline{V}(y; \theta) = \theta_0 + \sum_{t=1}^T \sum_{i \in \mathcal{N}_\text{SV} \cap \mathcal{N}_\text{RV}}\overline{V}_i^t(y_i^{-,t} - y_i^{+,t}; \theta_i^t) \, ,
\end{equation}
in which $\theta := (\theta_0, \theta_1^1, \ldots, \theta_{|\mathcal{N}_\text{SV} \cap \mathcal{N}_\text{RV}|}^T)$ is a vector of parameters, with $\theta_0 \in \mathbb{R}$ and $\theta_i^t \in \mathbb{R}^{2\overline{y}}$ for each index $(i, t)$. 

Expressing the elements of each station's subvector as $\theta_i^t = ([\theta_i^t]_{-\overline{y}}, [\theta_i^t]_{-\overline{y}+1}, \ldots, [\theta_i^t]_{\overline{y}-1})$, each function $\overline{V}_i^t : [-\overline{y},\overline{y}] \rightarrow \mathbb{R}$ is of the form
\begin{equation} \label{eq:vbaritdef}
    \overline{V}_i^t(x; \theta_i^t) = \left\{ \begin{array}{ll} 
    -\sum_{y' = -\overline{y}}^{-1} [\theta_i^t]_{y'} \,, & x = -\overline{y}\,, \\
    -\sum_{y' = \lceil x \rceil-1}^{-1} [\theta_i^t]_{y'} + [\theta_i^t]_{\lceil x \rceil-1}(x - \lfloor x \rfloor) \, , \quad & -\overline{y} < x < 0\,, \\
    0 \, , \quad & x = 0\,, \\ 
    \sum_{y' = 0}^{\lfloor x \rfloor} [\theta_i^t]_{y'} + [\theta_i^t]_{\lfloor x \rfloor}(x - \lfloor x \rfloor) \, , \quad & 0 < x < \overline{y}\,, \\
    \sum_{y' = 0}^{\overline{y}-1} [\theta_i^t]_{y'}\, , & x = \overline{y}\,. 
    \end{array} \right.
\end{equation}
Thus, $\overline{V}_i^t$ is a piecewise linear function passing through the origin, and with the $2\overline{y}$ elements of $\theta_i^t$ specifying the slopes between its integer breakpoints. We recall that constant $\overline{y}$ is the largest magnitude of action $y_i^{+,t}$ or $y_i^{-,t}$ under consideration by the operator. The form of $\overline{V}_i^t(\,\cdot\,; \theta_i^t)$ is illustrated in Fig.~\ref{fig:vftheta}.

We motivate a further restriction on the form of $\overline{V}(y; \theta)$ with the following property of the true second-stage \ac{VF}: \\

\begin{prop}[Convex value function] \label{prop:convexvf}
The second-stage \ac{VF}, $\mathbb{E}[V(y,\xi)]$, is a convex function of $y$.
\end{prop}
\begin{proof}
By introducing epigraph variables to convert the piecewise linear functions $l_{i,j}^{t,k}(\, \cdot \, , \xi)$ into linear constraints, problem \eqref{eq:stage2} can be written in the form \[ V(y,\xi) = \min_x c(\xi)^\top x + d \quad \text{s.~t.}\,\, Ax \leq b(\xi) + Dy \, .\] Thanks to Proposition \ref{prop:stage2int} the integrality constraints are not required. For fixed $\xi$ this satisfies the convexity conditions of \cite[Prop.~2.1]{fiacco_convexity_1986}, where the parameter domain $S$ used in that result corresponds to the projection of the feasible set of \eqref{eq:stage1} onto $y$-space in our setting. Thus $V(y,\xi)$ is convex in $y$. The expectation over $\xi$ is simply the linear combination $\mathbb{E}[V(y,\xi)] = \sum_{\xi' \in \Xi}\left[ \mathbb{P}(\xi = \xi')\cdot  V(y,\xi') \right]$, and thus also convex.
\end{proof}

Thus we impose the condition that the slopes $[\theta_i^t]_{-\overline{y}}, \ldots, [\theta_i^t]_{\overline{y}-1}$ be non-decreasing, i.e., forming a convex function in each coordinate variable. Moreover we bound the slope magnitudes by a suitably large value $\theta^{\max}$. We write the set of feasible parameters compactly as
\begin{equation}
    \Theta := \left\{ \theta \, \left| \, \begin{array}{lll} -\theta^{\max} \leq {[\theta_i^t]}_{y'} \leq \theta^{\max} & y' = -\overline{y}, \ldots,\overline{y}-1, & \forall (i,t), \\ {\quad\quad[\theta_i^t]}_{y'} \geq {[\theta_i^t]}_{y'-1} & y' = -\overline{y}+1, \ldots, \overline{y}-1, & \forall (i,t) \end{array} \right. \right\} \, .
\end{equation} 

\begin{figure}[t]
\centering
\begin{tikzpicture}[scale=1]

	\draw[myPath] (-4.5, 0) to (4.5, 0);
	\draw[myPath] (0, -0.5) to (0, 4);
	
	\draw[-] (4, 3.7) to (3, 2.4) to (2, 1.2) to (1, 0.4) to (0, 0) to (-1, -0.2) to (-2, -0.3) to (-3, 0.1) to (-4, 1.2);
	\draw[-] (4, 3.6) to (4, 3.8); \draw (3, 2.3) to (3, 2.5); \draw (2, 1.1) to (2, 1.3);
	\draw[-] (1, 0.3) to (1, 0.5); \draw (-1, -0.3) to (-1, -0.1); \draw (-2, -0.4) to (-2, -0.2);
	\draw[-] (-3, 0.0) to (-3, 0.2); \draw[-] (-4, 1.1) to (-4, 1.3);
	\draw[-, dashed] (2.25, 1.5) to (2.75, 1.5) to (2.75, 2.1);
	\draw node[right] at (2.75, 1.8) {$[\theta_i^t]_{\lfloor x \rfloor}$};
	
	\draw[-] (-4, -0.1) to (-4, 0.1); \draw[-] (-3, -0.1) to (-3, 0.1); \draw[-] (-2, -0.1) to (-2, 0.1);
	\draw[-] (-1, -0.1) to (-1, 0.1); \draw[-] (1, -0.1) to (1, 0.1); \draw[-] (2, -0.1) to (2, 0.1);
	\draw[-] (3, -0.1) to (3, 0.1); \draw[-] (4, -0.1) to (4, 0.1);
	\draw[-] (-0.1, 1) to (0.1, 1); \draw[-] (-0.1, 2) to (0.1, 2); \draw[-] (-0.1, 3) to (0.1, 3);

	\draw node[right] at (4.5, 0) {$x$};
	\draw node[left] at (0, 4) {$\overline{V}_i^t(x; \theta_i^t)$};
	
	\draw node[below] at (-4, 0) {$-\overline{y}$};
	\draw node[below] at (4, 0) {$\overline{y}$}; 
	\draw node[below right] at (0, 0) {$0$};
	
\end{tikzpicture}
\vspace{0.25cm}
\caption{Illustration of the convex value function component $\overline{V}_i^t(\,\cdot\,; \theta_i^t)$ defined in equation \eqref{eq:vbaritdef}. The elements of the vector $\theta_i^t$ specify the $2\overline{y}$ slopes between integer breakpoints, and the function always passes through the origin.}
\label{fig:vftheta}
\end{figure}
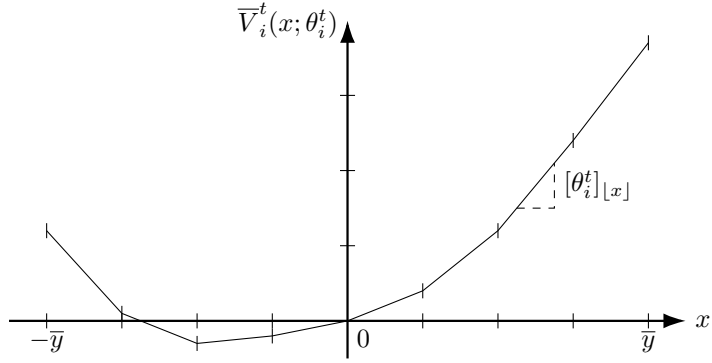

Problem \eqref{eq:stage1} can now be solved approximately by substituting $\overline{V}$ into the objective in place of $\mathbb{E}_\xi[V(y,\xi)]$:
\begin{equation} \label{eq:stage1_approx}
\begin{aligned}
    \min_{z,y} \quad & \sum_{t=1}^T \! \left[ \sum_{(i,j) \in \mathcal{E}_\text{RV}} \!\!\! c_{i,j}^t z_{i,j}^t + \sum_{i \in \mathcal{N}_\text{SV} \cap \mathcal{N}_\text{RV}}  \!\!\! r_i^t (y_i^{+,t} + y_i^{-,t}) \right] + \overline{V}(y; \theta) \\
    \text{s.~t.}\quad & \text{\eqref{eq:DRRPSVC2}-\eqref{eq:DRRP2Inta}}
\end{aligned}
\end{equation}
 We note immediately that since $\theta_0$ appears as a constant offset in the objective of \eqref{eq:stage1_approx}, it plays no role in the resulting first-stage decisions. We therefore estimate only the elements $\theta_i^t$, and note that an estimate of $\theta_0$ would be straightforward to obtain with a minor adaptation of the algorithm similar to \cite[\S 3]{powell_learning_2004}.

\subsection{Stochastic approximation algorithm}

We now describe an iterative procedure based on the so-called SPAR (\textit{separable, projective, approximation routine}) \cite{powell_learning_2004} to estimate the parameterization $\theta^\star \in \Theta$, approximating the true expected value of \eqref{eq:stage2}, that yields the most efficient solution to the two-stage problem \eqref{eq:stage1}-\eqref{eq:stage2}. The procedure is listed in Algorithm \ref{alg:spar}.

\begin{algorithm}[t]
	\caption{SPAR for dynamic rebalancing of shared mobility systems}\label{alg:spar}
	\textbf{Input:} $T$, $\Xi$, $(\mathcal{V}_v, \mathcal{E}_\text{RV}),(\mathcal{V}_c, \mathcal{E}_\text{SV})$, $n_{\max}$; mappings $f_{i,j}^{t,k}(\xi)$, $l_{i,j}^{t,k}(\,\cdot\,,\xi)$, $\alpha(n)$\\
	\textbf{Output:} First-stage actions $(z^{\rm final}, y^{\rm final})$, final parameter vector $\theta^{(n_{\max})}$ \\
	\textbf{Indices:} Iteration $n$, time step $t$, edges $(i,j) \in \mathcal{E}_\text{SV}$, durations $k$
	\begin{algorithmic}[1]
		\State $\theta^{(1)} \leftarrow \mathbf{0}$
        \For{$n=1,\ldots, n_{\max}$}
        \State Update cost function of \eqref{eq:stage1_approx} with approx.~cost-to-go $\overline{V}(y;\theta^{(n)})$
		\State $(z^{(n)}, y^{(n)}) \leftarrow $ Solve \eqref{eq:stage1_approx}, optionally with some or all integrality constraints relaxed \label{algl:stage1}
		\State Draw a new independent sample $\xi^{(n)} \in \Xi$ \label{algl:samplexi}
		\For{each tuple $(i,j,t,k)$}
		    \State Update demand $f_{i,j}^{t,k}(\xi^{(n)})$
		    \State Update journey valuation function $l_{i,j}^{t,k}(\, \cdot \, ; \xi^{(n)})$
		\EndFor
		\State Update RHS of problem \eqref{eq:stage2} with data $y^{(n)}$ \label{algl:yupdate}
		\State $\lambda^{(n)} \leftarrow$ Solve convex relaxation of \eqref{eq:stage2} \Comment{Multipliers for constraints \eqref{eq:DRRPNC2}} \label{algl:stage2}
		\State Construct gradient vector $\zeta(\lambda^{(n)})$ \label{algl:grad}
		\State $\tilde{\theta} \leftarrow \theta^{(n)} - \alpha(n) \zeta(\lambda^{(n)})$ \Comment{Gradient step}
		\State $\theta^{(n+1)} \leftarrow \arg \min_{\theta \in \Theta}\tfrac{1}{2}||\theta - \tilde{\theta}||^2$ \Comment{Project onto admissible param.~set}
		\EndFor
	\If{relaxed problem solved on line \ref{algl:stage1}}
	\State $(z^{\rm final}, y^{\rm final}) \leftarrow $ Re-solve \eqref{eq:stage1_approx} with all integrality constraints enforced \label{algl:final_integer}
	\Else
	\State $(z^{\rm final}, y^{\rm final}) \leftarrow (z^{(n_{\max})}, y^{(n_{\max})}) $
	\EndIf
	\end{algorithmic}
	\Return $(z^{\rm final}, y^{\rm final}, \theta^{(n_{\max})})$ \Comment{Final integer-feasible operator decision and approximate \ac{VF}}
\end{algorithm}

The superscript $(n)$ indicates a variable's value at iteration $n$. At each iteration, the algorithm uses the result of a sample instance of the second-stage problem \eqref{eq:stage2} to modify the gradient of $\overline{V}(y)$ at the last value of $y$ chosen in the first stage. The modification is determined by the sensitivity of the optimal value of \eqref{eq:stage2} to the first-stage decisions, which is readily obtained via the optimal dual variables for constraints \eqref{eq:DRRPNC2}.

The gradient vector $\zeta$ in line \ref{algl:grad} has the same dimension as $\theta$. Using $[\zeta_i^t]_y$ to denote the element of $\zeta$ corresponding to $[\theta_i^t]_y$, it is defined by 
\begin{equation}
    [\zeta_i^t]_{y'} = \left\{ \begin{array}{ll} \lambda_i^{+,t} - \lambda_i^{-,t} & \text{if $y' = y_i^{(n)\,-,t} - y_i^{(n)\,+,t}$,} \\ 0 & \text{otherwise,} \end{array} \right.
\end{equation}
where $\lambda_i^{+,t}$ and $\lambda_i^{+,t}$ are optimal dual variables for the upper and lower bounds of constraint \eqref{eq:DRRPNC2} respectively, and $y_i^{(n)\,+,t}$ and $y_i^{(n)\,-,t}$ are outputs from stage 1. The step size $\alpha(n)$ diminishes over iterations, and the asymptotic  convergence results derived in \cite{powell_learning_2004} rely on this rule satisfying $\alpha(n) \in (0,1] \, \, \forall n$, $\sum_{n=1}^\infty \alpha(n) = \infty$, and  $\sum_{n=1}^\infty (\alpha(n))^2 < \infty$.\footnote{The original formulation also allows for random step sizes, with slightly different requirements on these for convergence.}

\subsection{Use of continuous relaxations}

Thanks to the integer breakpoints in our parameterized approximator $\overline{V}(y;\theta)$, we can guarantee that integer-valued \ac{RV} routing decisions $z$ result in an integer-valued solution to the whole problem, even when the integrality constraints on $y$ and $b$ are not enforced. This may help to explain the relatively low computation times reported in Section \ref{sec:NR}. We prove this in the following Lemma and Proposition.\\

\begin{lem} \label{lem:minimal_y}
There always exists an optimizer of \eqref{eq:stage1_approx} with $y$ variables taking the form $(y_i^{+,t},0)$ or $(0,y_i^{-,t})$ for all tuples $(i,t)$.
\end{lem}
\begin{proof} The result follows from the same rationale as Lemma \ref{lem:minimal_y_orig}.
\end{proof}

\begin{prop} \label{prop:ints1}
If \eqref{eq:stage1_approx} remains feasible when $z$ is fixed to some integer-valued $z^\star$, then this fixed problem has an integer-valued optimal $b^\star$ and $y^\star$, even when integrality constraint \eqref{eq:DRRP2Int2a} is relaxed.
\end{prop}
\begin{proof}
With $z = z^\star$ fixed, the optimization over the remaining variables can be written as a min-cost network flow problem on a time-expanded graph, in which the flows $b_{i,j}^t$ traverse time steps and each node $(i,t)$ has $\overline{y}$ incoming and $\overline{y}$ outgoing flows whose sums represent $y_i^{-,t}$ and $y_i^{+,t}$ respectively. Thanks to the integer breakpoints of each function $\overline{V}_i^t(\cdot;\theta)$, and Lemma \ref{lem:minimal_y}, the objective function of \eqref{eq:stage1_approx} can be modelled exactly using an artificial construction of unit-capacity edges, sources and sinks, as shown in Fig.~\ref{fig:zintsol}. Thus, $(b^\star,y^\star)$ can be mapped to the solution of an equivalent min-cost flow problem with integer edge capacities, which (as in Proposition \ref{prop:stage2int}) has an integer-valued solution.
\end{proof}
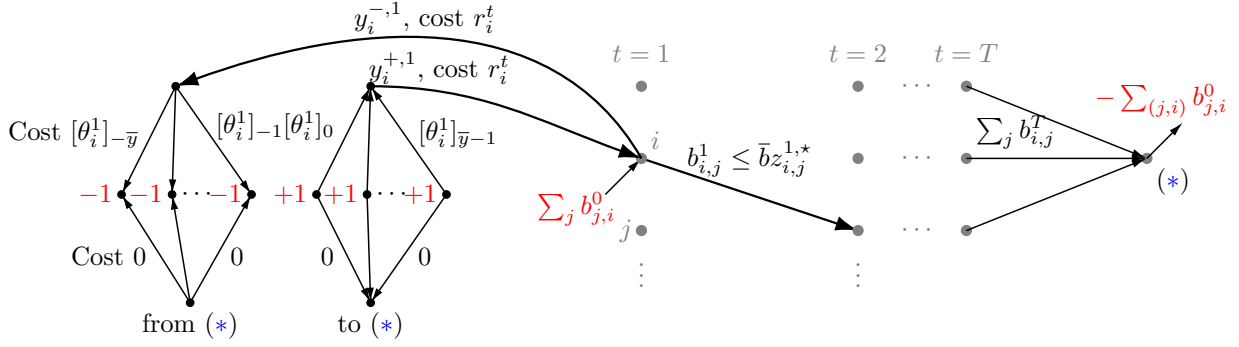
\begin{figure}[t]
\centering
\begin{tikzpicture}[scale=0.96]
	\coordinate (c11) at (0,3); \coordinate (c21) at (3,3); \coordinate (cT1) at (4.5,3);
	\coordinate (c12) at (0,2); \coordinate (c22) at (3,2); \coordinate (cT2) at (4.5,2);
	\coordinate (c13) at (0,1); \coordinate (c23) at (3,1); \coordinate (cT3) at (4.5,1);
	
	\coordinate (sink) at (7,2);

	\draw[myPath] (c12) -- (c23) node [midway, above, yshift=0.1cm] {$b_{i,j}^1 \leq \overline{b}z_{i,j}^{1,\star}$};
	\draw[myPath2] (c12){}+(-0.5,-0.5) -- (c12);
	\draw[red] (c12){}+(-0.9,-0.7) node {$\sum_j b_{j,i}^0$};

	\filldraw[gray] (c11) circle (2pt) node[above, yshift=0.2cm] {$t=1$};
	\filldraw[gray] (c21) circle (2pt) node[above, yshift=0.2cm] {$t=2$} node[right, xshift=0.45cm] {$\cdots$};
	\filldraw[gray] (c12) circle (2pt) node[above right] {$i$};
	\filldraw[gray] (c22) circle (2pt) node[right, xshift=0.45cm] {$\cdots$};
	\filldraw[gray] (c13) circle (2pt) node[below, yshift=-0.1cm] {$\vdots$} node[left] {$j$};
	\filldraw[gray] (c23) circle (2pt) node[below, yshift=-0.1cm] {$\vdots$}  node[right, xshift=0.45cm] {$\cdots$};
	\filldraw[gray] (cT1) circle (2pt) node[above, yshift=0.2cm] {$t=T$};
	\filldraw[gray] (cT2) circle (2pt) node [above right] {\color{black} $\sum_j b_{i,j}^T$};
	\filldraw[gray] (cT3) circle (2pt);
	\filldraw[gray] (sink) circle (2pt) node[below right] {\color{black}({\color{blue} $*$})};
	
	\draw[myPath2] (sink) to (7.5, 2.5);
	\draw[red] (sink){}+(0.2, 0.8) node {$-\sum_{(j,i)} b_{j,i}^0$};
	\draw[myPath2] (cT2) -- (sink);
	\draw[myPath2] (cT1) -- (sink);
	\draw[myPath2] (cT3) -- (sink);
	
	\coordinate (ytop1) at (-6.45, 3); \coordinate (ybottom1) at (-6.25, 0);
	\coordinate (y1) at (-7.2,1.5); 
	\coordinate (y2) at (-6.5,1.5);
	\coordinate (y3) at (-5.4,1.5); 
	\coordinate (ytop2) at (-3.75, 3); \coordinate (ybottom2) at (-3.75, 0);
	\coordinate (y4) at (-4.5,1.5);
	\coordinate (y5) at (-3.8,1.5); 
	\coordinate (y6) at (-2.7,1.5);
	
	\filldraw[black] (ytop1) circle (1.5pt);
	\filldraw[black] (ybottom1) circle (1.5pt) node[below] {from ({\color{blue} $*$})};
	\filldraw[black] (ytop2) circle (1.5pt);
	\filldraw[black] (ybottom2) circle (1.5pt) node[below] {to ({\color{blue} $*$})};
	\filldraw[black] (y1) circle (1.5pt) node[left] {\color{red}$-1$};
	\filldraw[black] (y2) circle (1.5pt) node[right, xshift=0.0cm] {$\cdots$} node[left] {\color{red}$-1$};
	\filldraw[black] (y3) circle (1.5pt) node[left] {\color{red}$-1$};
	\filldraw[black] (y4) circle (1.5pt) node[left] {\color{red}$+1$};
	\filldraw[black] (y5) circle (1.5pt) node[right, xshift=0.0cm] {$\cdots$} node[left] {\color{red}$+1$};
	\filldraw[black] (y6) circle (1.5pt) node[left] {\color{red}$+1$};
	
	\draw[myPath2] (ytop1) -- (y1) node [midway, left, yshift=0.1cm] {Cost $[\theta_i^1]_{-\overline{y}}$}; 
	\draw[myPath2] (ytop1) -- (y2); 
	\draw[myPath2] (ytop1) -- (y3)  node [midway, right, yshift=0.2cm, xshift=-0.1cm] {$[\theta_i^1]_{-1}$};
	\draw[myPath2] (ybottom1) -- (y1) node [midway, left, yshift=-0.1cm] {Cost $0$}; 
	\draw[myPath2] (ybottom1) -- (y2); 
	\draw[myPath2] (ybottom1) -- (y3) node [midway, right, yshift=-0.1cm] {$0$};
	
	\draw[myPath2] (y4) -- (ybottom2) node [midway, left, yshift=-0.1cm] {$0$}; 
	\draw[myPath2] (y5) -- (ybottom2); 
	\draw[myPath2] (y6) -- (ybottom2) node [midway, right, yshift=-0.1cm] {$0$};
	\draw[myPath2] (y4) -- (ytop2) node [midway, left, yshift=0.2cm, xshift=0.3] {$[\theta_i^1]_{0}$}; 
	\draw[myPath2] (y5) -- (ytop2); 
	\draw[myPath2] (y6) -- (ytop2) node [midway, right, yshift=0.1cm] {$[\theta_i^1]_{\overline{y}-1}$};
	
	\draw[myPath] (ytop2) to [out=0, in=160] (c12);
	\draw (-2.8, 3.6) node [below] {$y_i^{+,1}$, cost $r_i^t$};
	\draw[myPath] (c12) to [out=120, in=20] (ytop1);
	\draw (-3, 4.3) node [below] {$y_i^{-,1}$, cost $r_i^t$};

\end{tikzpicture}
\caption{Construction of integer-capacity min-cost network flow for \acp{SV} with fixed \ac{RV} routes $z^\star$ in Proposition \ref{prop:ints1}. Red values indicate source rates, with negative values indicating sinks. Black edge labels indicate the variable modelled by the flow on that edge, except where labelled as a cost. The construction on the left of the figure is repeated for each $(y_i^{+,t}, y_i^{-,t})$ pair, and is connected to the sink node ({\color{blue} $*$}) at the far right as shown. Each node in the $t=1$ column has a source flow $\sum_{j} b_{j,i}^0$ shown, and these sources are balanced by an equivalent sink at node ({\color{blue} $*$}). The in- and out-flows at each grey node reflect the conservation constraint \eqref{eq:DRRPSVC2}.}
\label{fig:zintsol}
\end{figure}





Despite this result, the computational bottleneck of needing to solve an integer program over routing decisions $z$ persists in line \ref{algl:stage1} of Algorithm \ref{alg:spar}. By dropping the integrality constraints \eqref{eq:DRRP2Int2a} and \eqref{eq:DRRP2Inta} one obtains a relaxation, allowing non-integral \ac{RV} flows, that can be solved more quickly. However, this overstates the flexibility of real-world \acp{RV}, and in general does not lead to an implementable (i.e.~integer) solution $(z^{(n_{\max})}, y^{(n_{\max})})$. Therefore one needs to plug the approximate value function achieved after $n_{\max}$ iterations into a final integer program \eqref{eq:stage1_approx} respecting the original constraints (line \ref{algl:final_integer} of Algorithm \ref{alg:spar}). Although there is no guarantee of the quality of this final solution, one can at least expect the end result to be obtained faster than solving integer programs throughout the algorithm. Another compromise is to enforce integrality on only a subset of the $z$ variables, for example in the first half of the planning horizon. These methods are compared in terms of solution quality and computation time in Section \ref{sec:NR}.

\subsection{Convergence properties} \label{sec:conv_prop}

The two-stage approach we use is motivated by the results of Powell \textit{et al.}~\cite{powell_learning_2004}, who in that study proved and discussed convergence of SPAR in increasingly complex settings. First, a single value function of the form \eqref{eq:vbaritdef} was considered, albeit passing through the origin at the edge of the domain rather than at the midpoint. It was shown \cite[Thm.~1]{powell_learning_2004} that if, at each iteration of an approximation algorithm, an unbiased estimate of a segment's gradient $[\theta_i^t]_{y'}$ can be obtained, then as long as the probability of visiting each segment has a strictly positive lower bound in the limit, the estimate converges to the true value function almost surely.

Second, the case of choosing the visited segments of a multi-dimensional but separable function as a result of an outer problem was considered. This violates the strictly-positive probability condition described above, because the optimizer will typically visit only a few segments infinitely often. But under a technical stability condition it can be shown \cite[Thm.~3]{powell_learning_2004} that an accumulation point of the algorithm solves the outer problem in which the value function appears.

Lastly, the authors turn to the case corresponding most closely to our formulation, where a non-separable value function is represented by a sum of separate functions of its coordinate variables. Although the above convergence results no longer apply, they report on a non-separable two-stage stochastic program for which SPAR still produces high-quality solutions. 

Although the SPAR algorithm is agnostic to whether or not the true \ac{VF} to be approximated is convex, each iteration includes a projection in parameter space onto the set $\Theta$ of convex function approximators. One might therefore expect a highly non-convex true \ac{VF} to cause convergence issues because of this. In our case though, as shown in Proposition 2, the true \ac{VF} is indeed convex, and our experience of benign computational behaviour agrees with that reported in \cite{powell_learning_2004}.


\section{Numerical simulations} \label{sec:NR}

We now evaluate the performance of the two-stage approach. In Section \ref{sec:unc_model} we describe the treatment of uncertainty in the model, then test our approximation scheme on two sets of test networks. Networks in the first set, described in Section \ref{sec:test_networks}, are artificially constructed on square grids. Customer demand for \acp{SV} is clustered such that the system enjoys only a moderate service rate unless rebalancing interventions are made by \acp{RV}. We report results on these networks in Section \ref{sec:num_results}. Then in Section \ref{sec:nyc_network} we move on to a case study constructed from public data for Philadelphia's public scheme.

\subsection{Uncertainty model} \label{sec:unc_model}

In all our numerical experiments we model the random arrival of customers wishing to start journeys with an \ac{SV} as independent events for start-end pairs $(i,j) \in \mathcal{E}_\text{SV}$, journey durations $k$, and time steps $t$. The values of these potential journeys are also considered independent. Thus, the functions $f_{i,j}^{t,k}(\xi)$ and $l_{i,j}^{t,k}(\,\cdot\,,\xi)$ could be viewed as independent mappings $f_{i,j}^{t,k}(\xi_{i,j}^{t,k})$ and $l_{i,j}^{t,k}(\,\cdot\,,\xi_{i,j}^{t,k})$ from uncorrelated sub-vectors of $\xi$.

In these experiments we make the standard assumption that events where customers arrive and attempt to start a journey are exponentially distributed in continuous time with known rate parameter. Thus each $f_{i,j}^{t,k}(\xi_{i,j}^{t,k}) \in \mathbb{N}$ follows a separately-parameterized Poisson distribution for each tuple $(i,j,t,k)$ in our discrete time setting. We assume each sampled sub-vector $\xi_{i,j}^{t,k}$ both

\begin{enumerate}
    \item[1)] determines the demand level $f_{i,j}^{t,k}(\xi_{i,j}^{t,k})$, and
    \item[2)] parameterizes the associated loss function $l_{i,j}^{t,k}(\,\cdot\,, \xi_{i,j}^{t,k})$.
\end{enumerate} 

The latter uses $f_{i,j}^{t,k}(\xi_{i,j}^{t,k})$ samples from a uniform distribution $U(l_{\min}, l_{\max})$, with $0 \leq l_{\min} \leq l_{\max}$, sorting the values in ascending order to obtain the convex, piecewise affine loss function described in Section \ref{sec:defs}.\footnote{Although conceptually in Algorithm \ref{alg:spar} the dimension of $\xi$ should be fixed \textit{a priori}, each loss function $l_{i,j}^{t,k}(\,\cdot\,, \xi_{i,j}^{t,k})$ has a random domain $\{0, \ldots, f_{i,j}^{t,k}(\xi_{i,j}^{t,k})\}$, and is therefore described by a random number of slope parameters. In practice, one can simply sample the required number of slopes of this function once $f_{i,j}^{t,k}$ has been realized, and view these samples as the first of an arbitrarily long sequence within $\xi_{i,j}^{t,k}$, whose length would be invariant to the realized value $f_{i,j}^{t,k}(\xi_{i,j}^{t,k})$.}


\subsection{Standardized test networks} \label{sec:test_networks}

We generate artificial, standardized test networks on square grids, in order to examine the scaling performance of our algorithm in a systematic manner. For these networks, the \ac{SV} movement graph $(\mathcal{N}_\text{SV}, \mathcal{E}_\text{SV})$ is fully connected whereas the \ac{RV} movement edges $\mathcal{E}_\text{RV}$, which always take one time step to traverse, are determined by Euclidean distance between the stations $\mathcal{N}_\text{SV}$ and the modelled vehicle speed. We let $\mathcal{N}_\text{RV} = \mathcal{N}_\text{SV}$ so that \acp{RV} have access to all \ac{SV} nodes for rebalancing actions. 

The clustering of demand is a primary cause of service degradation for shared mobility systems, and we model it in both the origin and destination of \ac{SV} journeys. For example, in morning rush hours where \acp{SV} are frequently used for the ``last mile'' of commutes, one encounters high demand for journey starts \acp{SV} at incoming commuter rail stations, and high demand for journey ends near workplaces. We use clustering because homogenous demand patterns under-represent the difficulty of typical real-world rebalancing problems. The generic demand generation procedure is detailed in Algorithm \ref{alg:demand}; we describe the particular parameter settings used in Section \ref{sec:num_results}.

\begin{algorithm}[t]
	\caption{Creation of artificial scenarios for clustered nominal demand on a rectangular grid}\label{alg:demand}
	\textbf{Input:} Number of orig.~and dest.~clusters $(O,D)$, number of time brackets $B$, time bracket length $T_B$, time horizon $T = BT_B$, number of \acp{SV} $N$, \ac{SV} speed $v$, step duration $\Delta t$\\
	\textbf{Output:} Nominal customer demand $F_{i,j}^{t,k}$ for $(i,j) \in \mathcal{E}_\text{SV}$, $t=1,\ldots,T$, $k = 0,\ldots,K$
	\begin{algorithmic}[1]	
		\State	$F_{i,j}^{t,k} = 0$ for all tuples $(i,j,t,k)$
		\For{Time bracket $\beta = 1$ to $B$}
			\State $\mathcal{O}_\beta = \emptyset$, $\mathcal{D}_\beta = \emptyset$, 
			\For{$o = 1$ to $O$}
			\State $(x,y) \leftarrow (\texttt{Uniform}(0,100), \texttt{Uniform}(0,100))$ \Comment{Centroid of trip origin cluster} \label{algl:Ostep1}
			\State $\Sigma \leftarrow \texttt{RandInt}(1,4) \cdot  \frac{100}{O} \cdot \binom{1 \,\, 0}{0 \,\, 1}$ \Comment{Covariance matrix of distribution}
			\State $\mathcal{O}_\beta \leftarrow \mathcal{O}_\beta \cup \{{\rm Normal}(\binom{x}{y},\Sigma)\}$ \Comment{Two-dimensional pdf of origin cluster} \label{algl:OstepN}
			\EndFor
			\For{$d = 1$ to $D$}
			\State Execute lines \ref{algl:Ostep1}-\ref{algl:OstepN} on $(\mathcal{D}_\beta, D)$ instead of $(\mathcal{O}_\beta,O)$ \Comment{Trip destination clusters}
			\EndFor
		
	\For{$t=(\beta - 1)T_B + 1$ to $\beta T_B$} \Comment{Time steps $t$ within time bracket $\beta$}
			\State $\overline{F}_t \leftarrow \texttt{RandomNumberOfTrips}(N, t)$ \label{algl:RandNTrips}
	\For{$f = 1$ to $\overline{F}_t$}
	\State $o \leftarrow \texttt{RandInt}(1,O)$, $d \leftarrow \texttt{RandInt}(1,D)$ \Comment{Choose indices of origin and destination clusters}
	\State Sample $(x_o,y_o)\sim \mathcal{O}_\beta[o], (x_d,y_d) \sim \mathcal{D}_\beta[d]$  \Comment{Sample start and end locations from pdfs}
	\State $(i,j) \leftarrow \texttt{MapToGrid}((x_o,y_o), (x_d,y_d))$ \Comment{Map to nearest node in $\mathcal{N}_{\rm SV}$}
	\State $k \leftarrow \texttt{CalcTime}((x_o,y_o), (x_d,y_d), v, \Delta t)$ \Comment{Trip duration in discrete time steps}
	\State $F_{i,j}^{t,k} \leftarrow F_{i,j}^{t,k} + 1$ \Comment{Update nominal demand for tuple $(i,j,t,k)$}
	\EndFor
	\EndFor
	\EndFor
	\end{algorithmic}
	\Return $F$
\end{algorithm}

The algorithm creates a fixed number of Gaussian peaks in lines \ref{algl:Ostep1}-\ref{algl:OstepN} representing concentrations of trip start locations and trip end locations. The peaks differ for different time brackets $\beta = 1,\ldots,B$, which each contain $T_B \geq 1$ time steps. At each time step $t$, it then determines a random number of ``desired trips'' for customers using the function $\texttt{RandomNumberOfTrips}(N, t)$, where $N$ is the number of \acp{SV} in the system, in line \ref{algl:RandNTrips}.  Each desired trip has an origin and destination sampled in continuous 2D space from randomly-indexed origin and destination distributions. This pair is then mapped to the nearest nodes $(i,j)$ in $\mathcal{N}_\text{SV}$ by the function $\texttt{MapToGrid}((x_o,y_o), (x_d,y_d))$. The duration $k$ is inferred from geometry and the \ac{SV} speed in the following line, after which the relevant indexed demand quantity $F_{i,j}^{t,k}$ is incremented. 

\begin{figure}[tp]
\begin{center}
\includegraphics[width=0.4\textwidth]{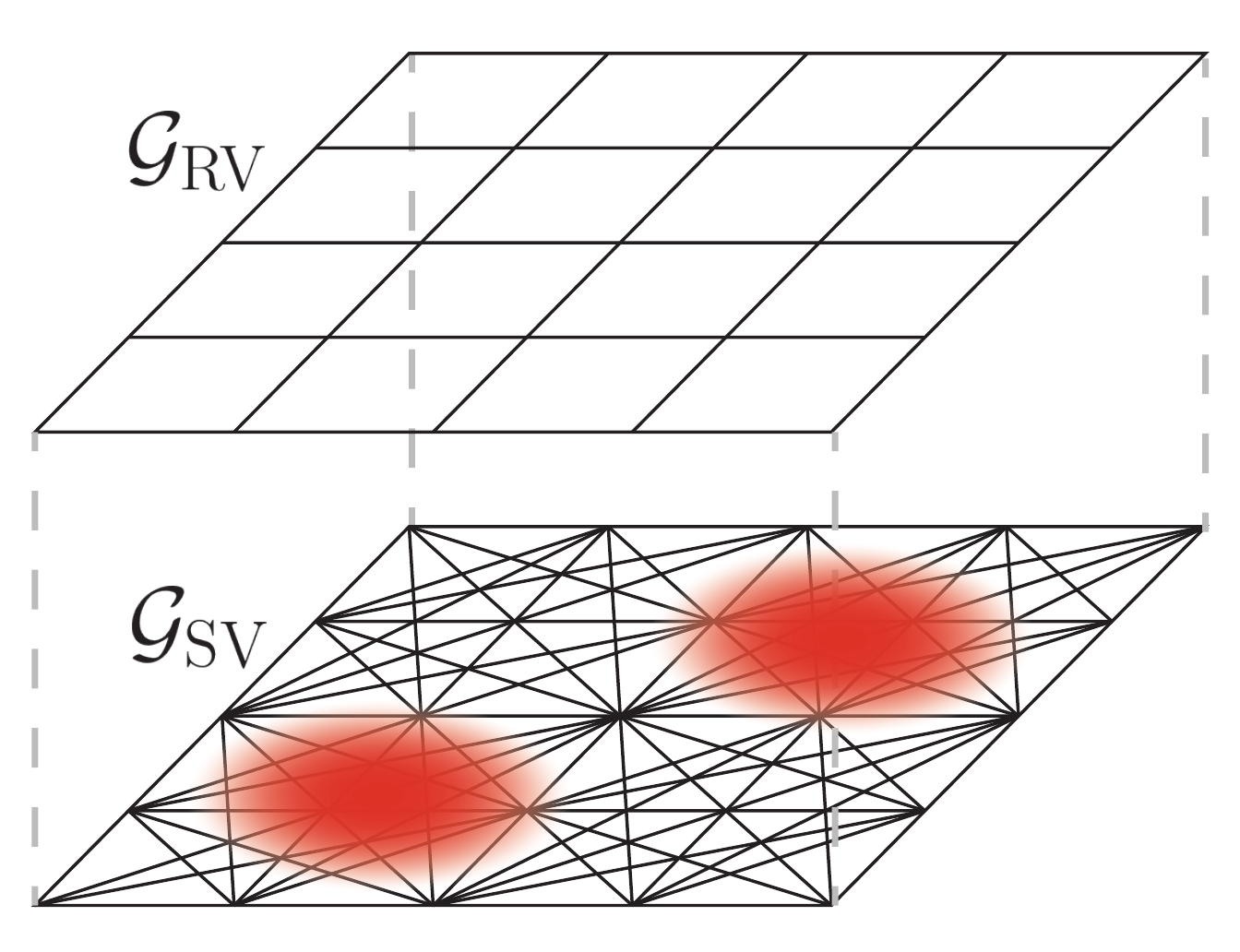}
\caption{Schematic of artificial network, showing randomly-generated demand concentrations on the \ac{SV} graph. In Algorithm \ref{alg:demand} the demand distribution in continuous space is sampled, and each sample is mapped to the nearest node $i \in \mathcal{N}_\text{SV}$.}
\label{fig:artificial_demand}
\end{center}
\end{figure}

\subsection{Numerical results} \label{sec:num_results}

\subsubsection{Algorithm parameters}

We implemented the following solution approaches. In each, the $(y,z)$ decisions from the first stage were made as described, and for the purpose of evaluation the objective \eqref{eq:stage1_obj} was estimated using a Monte Carlo simulation of the second stage cost via problem \eqref{eq:stage2Pen}, averaging over 100 random demand scenarios.
\begin{enumerate}
	\item[NA.] No action to control the system (no \ac{RV} movements and $y=0$).
	\item[M1.] SPAR, solving \eqref{eq:DRRP} for $y$ and $z$ under the assumption that in the second stage is deterministic, with customer demand $f_{i,j}^{t,k}$ and customer journey values $l_{i,j}^{t,k}(\cdot)$ set to expected values ($F_{i,j}^{t,k}$ rounded to the nearest integer and $l_{i,j}^{t,k}(x) = \tfrac{1}{2}(l_{\min} + l_{\max})x$ respectively).
    \item[M2-R.] SPAR as described in Algorithm \ref{alg:spar}, but using the \ac{LP} relaxation of problem \eqref{eq:stage1_approx}, i.e., non-integral \ac{RV} routing. After $n_{\max}$ iterations the non-relaxed version of \eqref{eq:stage1_approx} was solved to obtain a feasible solution.
    \item[M2-HI.] SPAR as described in Algorithm \ref{alg:spar}, but with integrality constraints on $z_{i,j}^t$ relaxed for $t > T/2$. As with M2-R, a non-relaxed version of Stage 1 was solved after the $n_{\max}$ ``half-relaxed'' iterations were completed.
    \item[M2-I.] SPAR as described in Algorithm \ref{alg:spar}, respecting integrality constraints throughout the algorithm.
    \item[M3.] SPAR but picking $y$ at random instead of solving problem \eqref{eq:stage1_approx} on line \ref{algl:stage1}. For each tuple $(i,t)$, the load/unload action was chosen with equal probability from $\{-\overline{y},-\overline{y}+1,\ldots,\overline{y}\}$. After $n_{\max}$ iterations, solve the first-stage problem \eqref{eq:stage1_approx} to obtain a final output as in Algorithm \ref{alg:spar}.
\end{enumerate}

We avoided any tuning of the basic algorithm, and used the step size rule suggested in \cite{powell_learning_2004}, namely $\alpha(n) = 20/(40 + n)$. The anecdotal evidence reported in \cite[Figs.~1 (a), (b)]{powell_learning_2004} suggests that SPAR may take on the order of only a few tens of iterations to approach a high-quality solution, which is faster than would be explained by the rate of decay of the step size $\alpha(n)$. In our simulations we had a similar experience despite solving a very different problem. We set $n_{\max} = 50$ in all cases, except for M3, where we used 200 iterations to obtain acceptable performance. The latter choice was made as M3 does not benefit from the simultaneous optimization and value function approximation effect enjoyed by the other methods.

\begin{table}[tp]
\caption{Number of variables and constraints in first-stage problem \eqref{eq:stage1_approx} for method M2-I. The numbers of integer and continuous variables are equal in each case, and independent of $|\mathcal{V}|$ owing to the flow formulation.}
\begin{center}
{\footnotesize
\begin{tabular}{l|rr|r|l}
\toprule
$|\mathcal{N}_\text{SV}|$ & Integer variables & Continuous variables & Constraints & $|\mathcal{V}|$ tested\\
\midrule
9 & 1188 & 1188 & 3456 & 1, 3\\
16 & 3456 & 3456 & 7488 & 1, 5\\
25 & 8100 & 8100 & 14,400 & 1, 5, 9\\
36 & 16,416 & 16,416 & 25,488 & 1, 5, 11\\
64 & 50,688 & 50,688 & 66,816 & 1, 9, 15\\
100 & 122,400 & 122,400 & 147,600 & 1, 9, 19\\
225 & 612,900 & 612,900 & 669,600 & 1, 13, 25\\
\bottomrule
\end{tabular}
}
\end{center}
\label{tab:num_vars}
\end{table}

\subsubsection{Network parameters}

Tests were performed with $|\mathcal{N}_\text{SV}| \in \{9,16,25,36,64,100,225\}$. Table \ref{tab:num_vars} lists the number of variables and constraints in each network's first stage problem \eqref{eq:stage1_approx} for method M2-I; some or all integer variables were relaxed in M2-HI and M2-R.  The numbers of \acp{RV} tested for each size are listed in the last column; note that due to the flow formulation, the number of model variables is invariant to $|\mathcal{V}|$. For each network size, 10 random instances of the artificial networks described in Section \ref{sec:test_networks} were created, with the parameters $O=3$, $D=5$, $B=6$, $T_B = 2$, $T=12$, $N=5|\mathcal{N}_\text{SV}|$, $v=125 / \sqrt{|\mathcal{N}_\text{SV}|}$ per time step, $\Delta t = $ 15 minutes.
The function $\texttt{RandomNumberOfTrips}(N, t)$ was configured to generate a rounded sample from a normal distribution with mean $0.15N$ and standard deviation $0.075N$, rejecting results below zero. Demand realizations were generated in line \ref{algl:samplexi} of Algorithm \ref{alg:spar} by sampling for each tuple $(i,j,t,k)$ from a Poisson distribution with the nominal demand $F_{i,j}^{t,k}$ as the rate parameter. 

\subsubsection{Mobility system parameters}

The first- and second-stage optimization were parameterized by $\overline{d}_i = 10$, $\overline{b}=5$, $\overline{y} = 10$, $d_i^0 = \overline{d}_i/2$, $b_{i,j}^t = 0$.  \acp{RV} were initialized on uniformly-weighted random nodes, and were fast enough to travel only to an adjacent node (excluding diagonals) in one time step, as depicted on the graph $\mathcal{G}_\text{RV}$ in Fig.~\ref{fig:artificial_demand}. It was assumed that no \acp{SV} were in transit at $t=0$, and the largest \ac{SV} trip duration modelled was $K = 2$ time steps. All journey values were assumed to be distributed uniformly between \$0.50 and \$1.50, regardless of origin, destination, or trip duration. \ac{RV} movement costs were set to $c_{i,j}^{t,k} = \$10^{-3}$ for $i \neq j$ and $0$ for $i=j$; load/unload costs were set to $r_i^t = \$10^{-3}$ for all $i$ and $t$. These penalties were set to very small values purely to discourage arbitrary uncosted \ac{RV} actions under the initial condition $\overline{V}(y;0) = 0$, and the main performance metric of interest was the social cost of unserved customer demand. The model feasibility penalty was set to $r_p =$ \$20 per \ac{SV} created or destroyed. The hardware used was an Intel Core i7 2.6 GHz CPU, and 16 GB RAM. All problems were solved in Gurobi 7.0.2, with computation confined to 2 CPU threads. 

\subsubsection{Performance evaluation}

For the purpose of illustrating the behaviour of the SPAR algorithm, Fig.~\ref{fig:examplev} shows an example of the \ac{VF} approximators returned by the algorithm.

The graphs that follow compare the behaviours of the variants tested. Figure \ref{fig:results_svc_rates} shows a graph of the solution quality in terms of service rate, with Tables \ref{tab:method_comparison_svc} and \ref{tab:method_comparison_cost} in Appendix B providing the underlying numerical values for service rate and objective value respectively.\footnote{We present results primarily in terms of service rate, despite using a dollar objective, as this has the benefit of being comparable across system sizes and is easier to interpret. Comparison of Tables \ref{tab:method_comparison_svc} and \ref{tab:method_comparison_cost} in Appendix B shows that the relative dominance of difference methods is unaffected, as our objective function is in effect very similar to service rate.} Figure \ref{fig:results_comp_times} shows average time to solve the Stage 1 problem \eqref{eq:stage1_approx} to an MIP tolerance of $5 \times 10^{-3}$ within M1 and M2 (Stage 1 is not solved in M3); the numbers are reproduced in Table \ref{tab:comp_times} in AppendixB. Solution times for Stage 2 are shown in the last column of Table \ref{tab:comp_times} for completeness; these are much lower than for Stage 1. 

\begin{figure}[tp]
\begin{center}
\includegraphics[width=\textwidth]{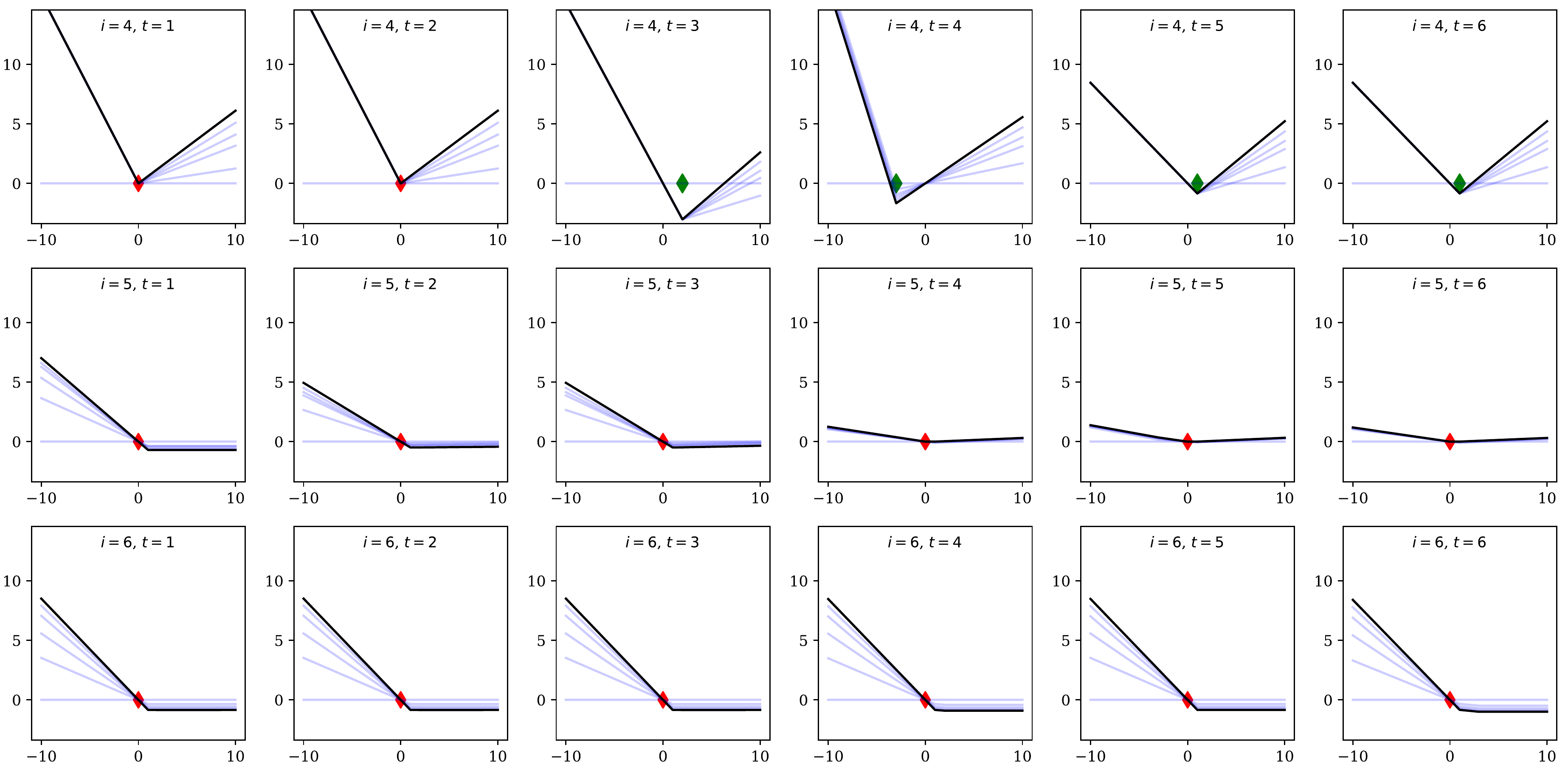}
\caption{Example of value function approximations $\overline{V}_i^t(y_i^{-,t} - y_i^{+,t})$, for $|\mathcal{N}_\text{SV}|=9$ and $|\mathcal{V}|=1$, for time steps $t=1$ to $6$ at stations $4$, $5$, and $6$. Each plot shows a different $(i,t)$ pair, with the black line showing the final estimate after 50 iterations, and lighter lines showing previous iterations. Markers on the horizontal axis denote load/unload decisions $y_i^{-,t} - y_i^{+,t}$ in the final solution to \eqref{eq:stage1_approx}; red indicates no action.}
\label{fig:examplev}
\end{center}
\end{figure}
%

\begin{figure}[tbp]
\begin{center}
\includegraphics[width=\textwidth]{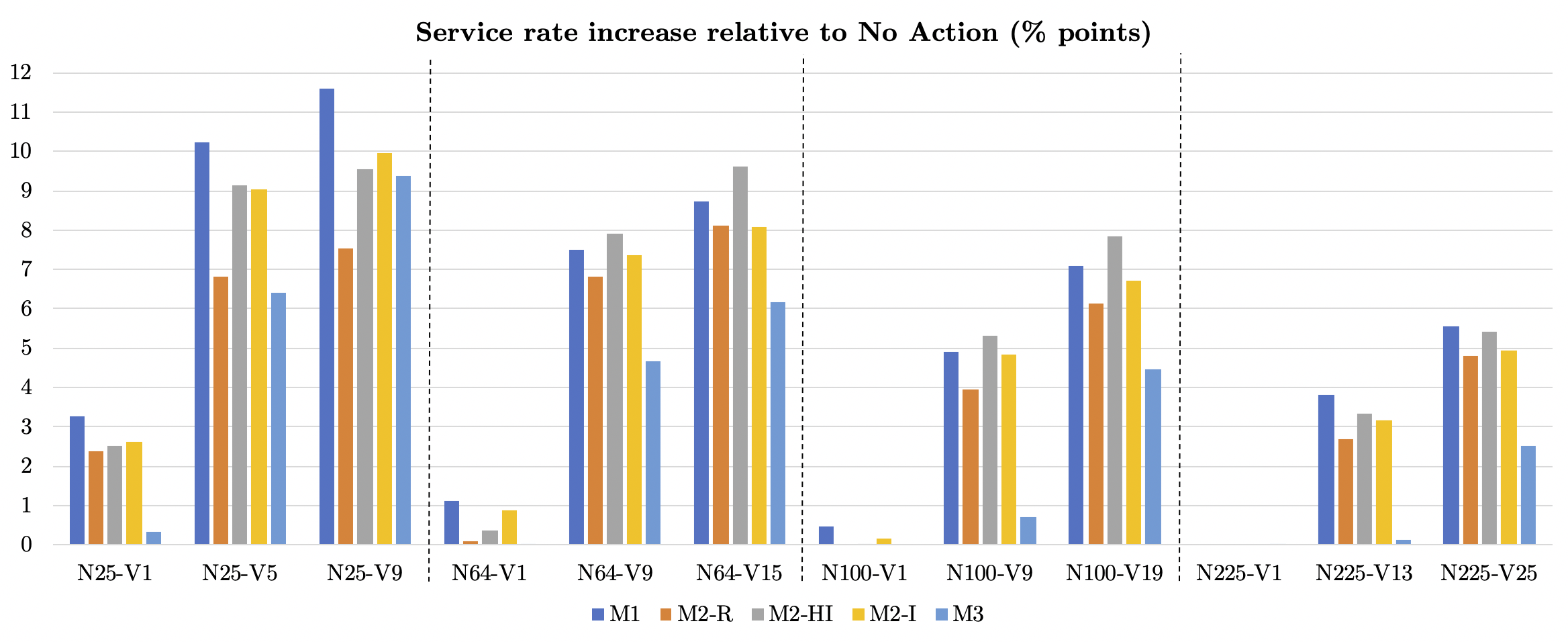}
\caption{Service rate improvement with respect to No Action by the system operator, for different system sizes $|\mathcal{N}_\text{SV}|$ and numbers of \acp{RV}, $|\mathcal{V}|$. Numerical values, including those for $|\mathcal{N}_\text{SV}| = 9, 16, 36$ are provided in Table \ref{tab:method_comparison_svc} in the Appendix. Similar results for objective value change induced are presented in Table \ref{tab:method_comparison_cost}.}
\label{fig:results_svc_rates}
\end{center}
\end{figure}

\begin{figure}[tbp]
\begin{center}
\includegraphics[width=\textwidth]{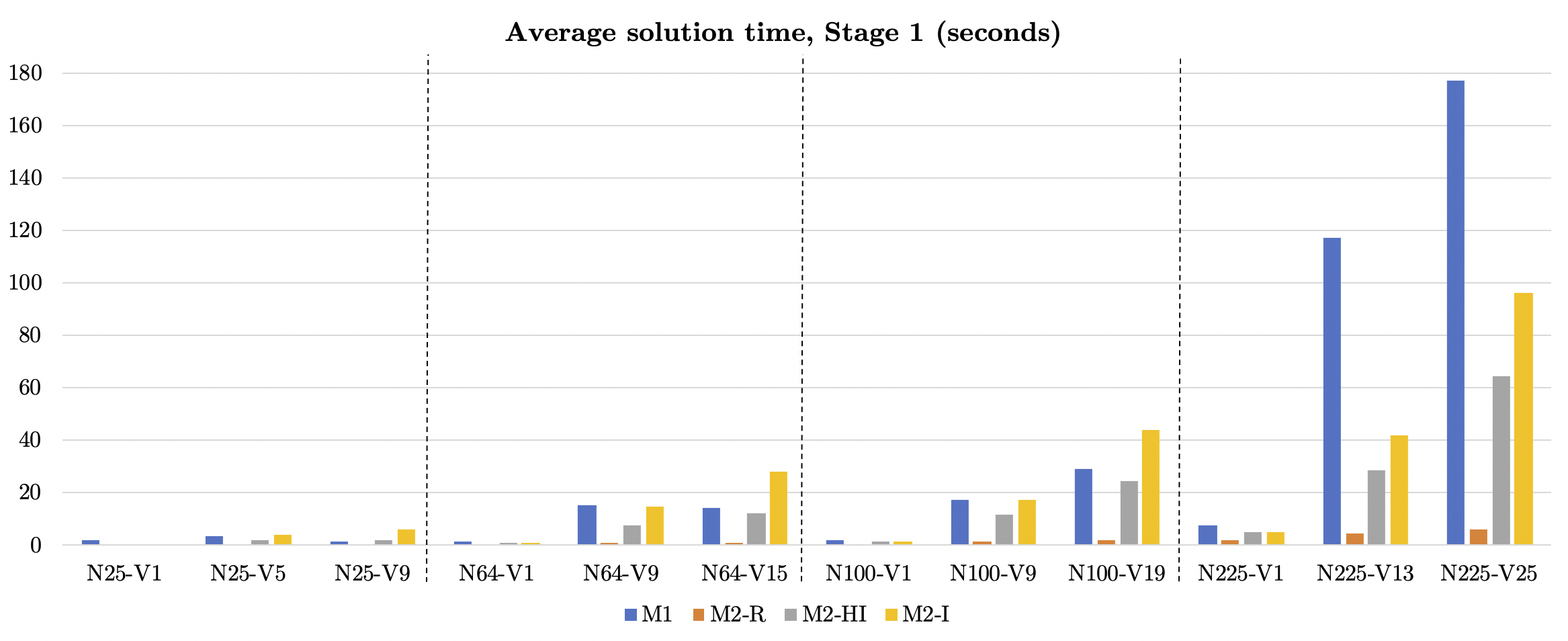}
\caption{Average computation per SPAR iteration spent solving Stage 1 problem \eqref{eq:stage1_approx} for different system sizes $|\mathcal{N}_\text{SV}|$ and numbers of \acp{RV}, $|\mathcal{V}|$. Although the deterministic method M1 obtained similar result quality to those in M2, its computation times grew substantially longer for the 225-node systems.}
\label{fig:results_comp_times}
\end{center}
\end{figure}

\begin{figure}[tbp]
\begin{center}
\includegraphics[width=\textwidth]{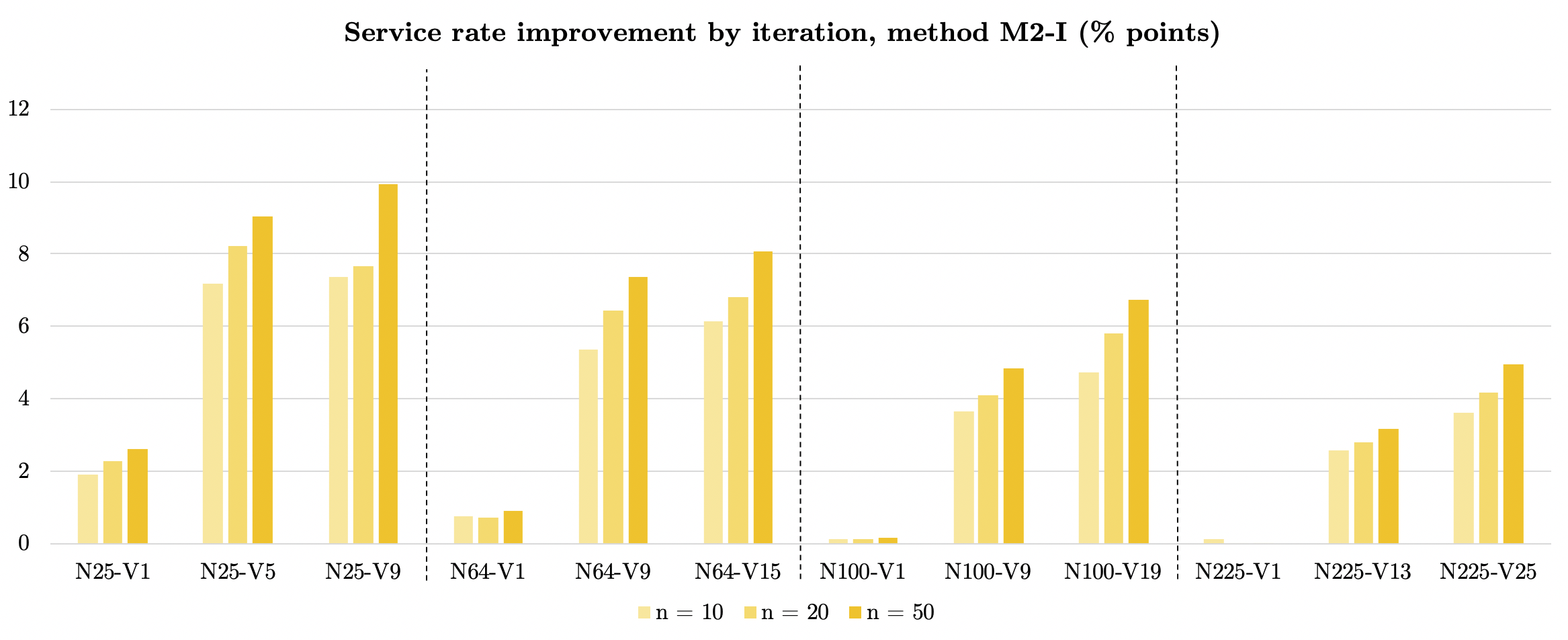}\\
%
\includegraphics[width=\textwidth]{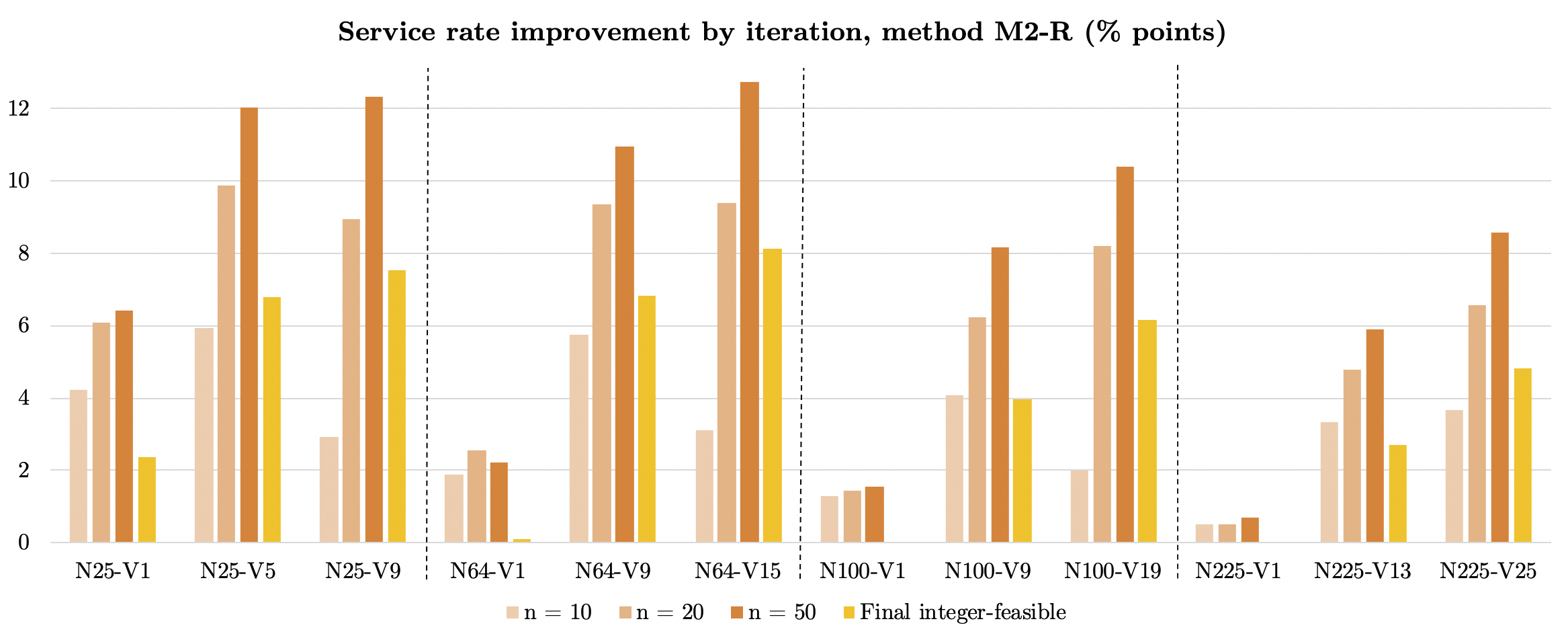}
\caption{Average service level increase by iteration for methods M2-I and M2-R. Service levels in M2-R appear to dominate those in M2-I, but this apparent advantage is lost when a solution respecting integrality constraints is generated after iteration 50.}
\label{fig:results_iters}
\end{center}
\end{figure}

\begin{figure}[tbp]
\begin{center}
\includegraphics[width=1.0\textwidth]{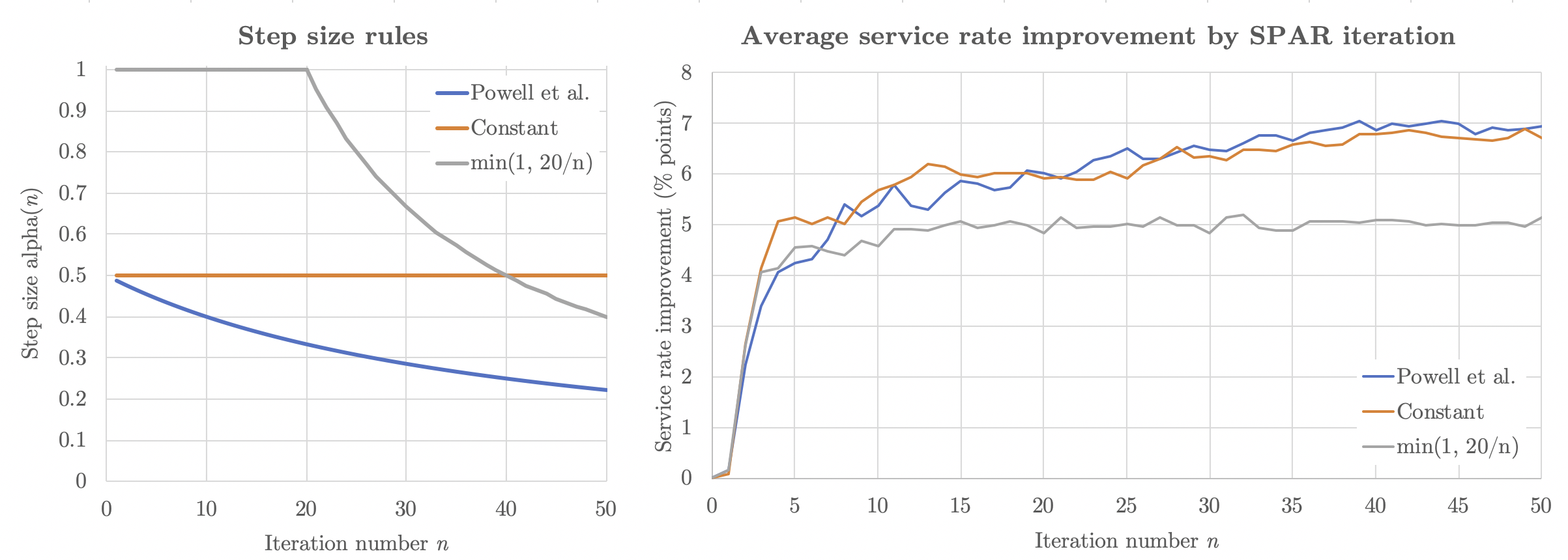}
\caption{Effect of two alternative step size rules on performance of SPAR over 50 iterations, under method M2-I.}
\label{fig:stepsizeeffect}
\end{center}
\end{figure}

On average, all variants caused improvements to the service rate (or almost equivalently in this setting, solution cost in dollars) after 50 iterations -- except where only one \ac{RV} was used on the largest networks. For each solution approach there is also a clear improvement to service rate as the number of \acp{RV} increases for a given network size. 

However, the computation times differ markedly between the methods. Although out of M1 and all M2 variants, M2-R typically generated slightly inferior results, the 50 iterations executed in a fraction of the time needed by the other methods.\footnote{Note that a full integer solution was still required after these ``relaxed'' iterations has been completed, as in line \ref{algl:final_integer} of Algorithm \ref{alg:spar}, which took an amount of time comparable to the each iteration of M2-I.} While the deterministic treatment of demand in M1 appears to lead to a similar or better final performance in comparison to M2 variants, computation times in stage 1 grew very long for the largest networks tested. Method M2-HI appears to offer competitive performance at significantly lower computational cost than the variants M1 and M2-I, which solve full integer problems at each iteration.

The performance of the solutions by iteration number $n$ is shown in Fig.~\ref{fig:results_iters} for methods M2-I and M2-R. In M2-I most of the service improvement was made in the first 10 or 20 iterations. As this variant generates an integer feasible solution at each iteration, this suggests that in a real-time control implementation it would often be safe to terminate the algorithm before all 50 iterations have occurred. In the case of M2-R, performance gains made up to iteration 50 are lost at the point where a final integer-feasible solution has to be generated; the performance estimates made during the iterations turn out to have been optimistic because the integrality constraints on \ac{RV} movements were not enforced. Thus the capability of the \acp{RV} to rebalance the system was overestimated. The underlying numbers are provided in Tables \ref{tab:artificial_results_int} and \ref{tab:artificial_results} in the Appendix.

Method M3 is different from the others, in that the first-stage solutions are obtained by random choice rather than solving problem \eqref{eq:stage1_approx}. This still appears to lead to worthwhile improvements in service rate, even without performing any computationally expensive optimization during the SPAR iterations. Although we used 200 rather than 50 iterations to generate reliable improvements in solution quality, the results suggest that for even larger systems, meaning $|\mathcal{N}_\text{SV}|$ in the high hundreds or the thousands, some kind of random sampling may be an attractive way to manage the scale of the problem. This is because the first-stage problem may become too large to solve in a real-time setting. However, apart from for very small systems, the solution quality returned by M3 was still somewhat worse than the other methods.

\subsubsection{Scalability of the algorithm} \label{sec:gran}

Thanks to the relatively tight relaxation properties of our formulation, we find that in M1 and M2, the solver is able to generate good integer-feasible solutions to these large problems in a short enough time for real-world operational use, meaning substantially less than 1 hour. Further changes to solver parameters, solution tolerances, and the relaxation of different subsets of the integrality constraints could yet improve the trade-off between computation time and solution quality.

It would be attractive for the algorithm we present to be applicable to all real-world systems. Some of these, for example the largest urban bike-sharing networks in use today, have more than 225 stations. We limited the size to 225 in the present paper due to the need to measure statistics across repeat experiments, as well as for each of the methods M1, M2-R, M2-HI, M2-I, and M3. However, in other tests we have successfully generated solutions for systems with 400 nodes, in times suitable for a real-time implementation of the algorithm. Our experience suggests that with additional work to reduce the variance of the Stage 1 solution times, our method could be extended to cover substantially larger systems.

Another motivation for improving the Stage 1 solution time is to allow a given real-world system to be modelled in a more granular manner, meaning more variables. Common to all multi-period optimization formulations, including \cite{contardo_balancing_2012} for bike sharing, we must map the duration of each real-world action to an integer number of time steps. Thus increased accuracy can be obtained by reducing the time discretization interval and including more time steps in the formulation. This is likely also to introduce additional spatial nodes for the shorter ``hops'' that \acp{SV} and \acp{RV} complete in the shorter time steps. The choice of temporal and spatial discretization is ultimately a trade-off between tractability and realism.

\subsubsection{Effect of step size rule $\alpha(n)$} \label{sec:step_size}

In the above experiments we used the step size rule $\alpha(n) = 20/(40+n)$ proposed in \cite{powell_learning_2004}, but it is instructive to examine the effect of alternative rules as a sensitivity analysis. Fig.~\ref{fig:stepsizeeffect} shows service rate results with method M2-I for $|\mathcal{N}_\text{SV}| = 64$ and $|\mathcal{V}| = 9$ under alternative step size rules, $\alpha(n) = 0.5$ and $\alpha(n) = \min\{1, 20/n\}$, as shown on the left plot. These were chosen as examples of functions with qualitatively different behaviour, but which are non-increasing and never take values larger than $1$. The resulting performance averaged over 10 problem instances is plotted on the right-hand side. While the rule from \cite{powell_learning_2004} had slightly superior performance on average, all three rules produced comparable improvements to service rate. Behaviour in early iterations was similar despite the rather different step sizes in effect during those iterations.

\subsection{Philadelphia case study} \label{sec:nyc_network}

Experiments were run on a model of the Philadelphia bike sharing system, with $|\mathcal{N}_\text{SV}| = 102$ docking stations and 1103 bikes, parameterized using public customer journey data from April 2015 to June 2017. The system is shown in Fig.~\ref{fig:philly_stations}. The methods described above were applied to a 3-hour time horizon containing $T=12$ steps of 15 minutes starting at 8am, with the customer demand data $F_{i,j}^{t,k}$ corresponding to rates measured for a June weekday. The maximum journey duration modelled was $K=4$ time steps, i.e., 1 hour. We assumed $|\mathcal{V}| = 15$ \acp{RV} (i.e., repositioning trucks) were present, each with capacity $\overline{d} = 10$, and assumed \acp{RV} could reach all customer stations, i.e.~$\mathcal{N}_\text{RV} = \mathcal{N}_\text{SV}$. The speed of the \acp{RV} travelling around the city was such that on average 16\% of stations were reachable in one time step from any other station. The parameterization of the methods was otherwise the same as in Section \ref{sec:num_results}. We performed 10 tests of the methods, corresponding to 10 different random initial locations of the \acp{RV}.

The average service rate under no rebalancing actions was 69.96\%, and the average customer demand was approximately 450 trips during the planning horizon. Results in terms of service rate, cost, and computation time are shown in Table \ref{tab:philly_stats}. Method M1 performed the worst, caused by the limitation of the deterministic, integer-valued representation of second-stage demand. For this system the average demand $F_{i,j}^{t,k}$ per 15 minute interval was below $0.5$ for most tuples $(i,j,t,k)$, and these were rounded down to $0$, meaning the \acp{RV} were scheduled against an underestimate of customer demand. The results using the relaxations in M2-R and M2-HI were superior to those using full integer solutions in M2-I. However, the final integer solution was much slower to generate, in several instances timing out at the pre-set limit of 1200 seconds. In cases where a timeout did occur, the suboptimality bound was between 1\% and 2\%.

\begin{figure}[tp]
\begin{center}
\includegraphics[width=0.6\textwidth]{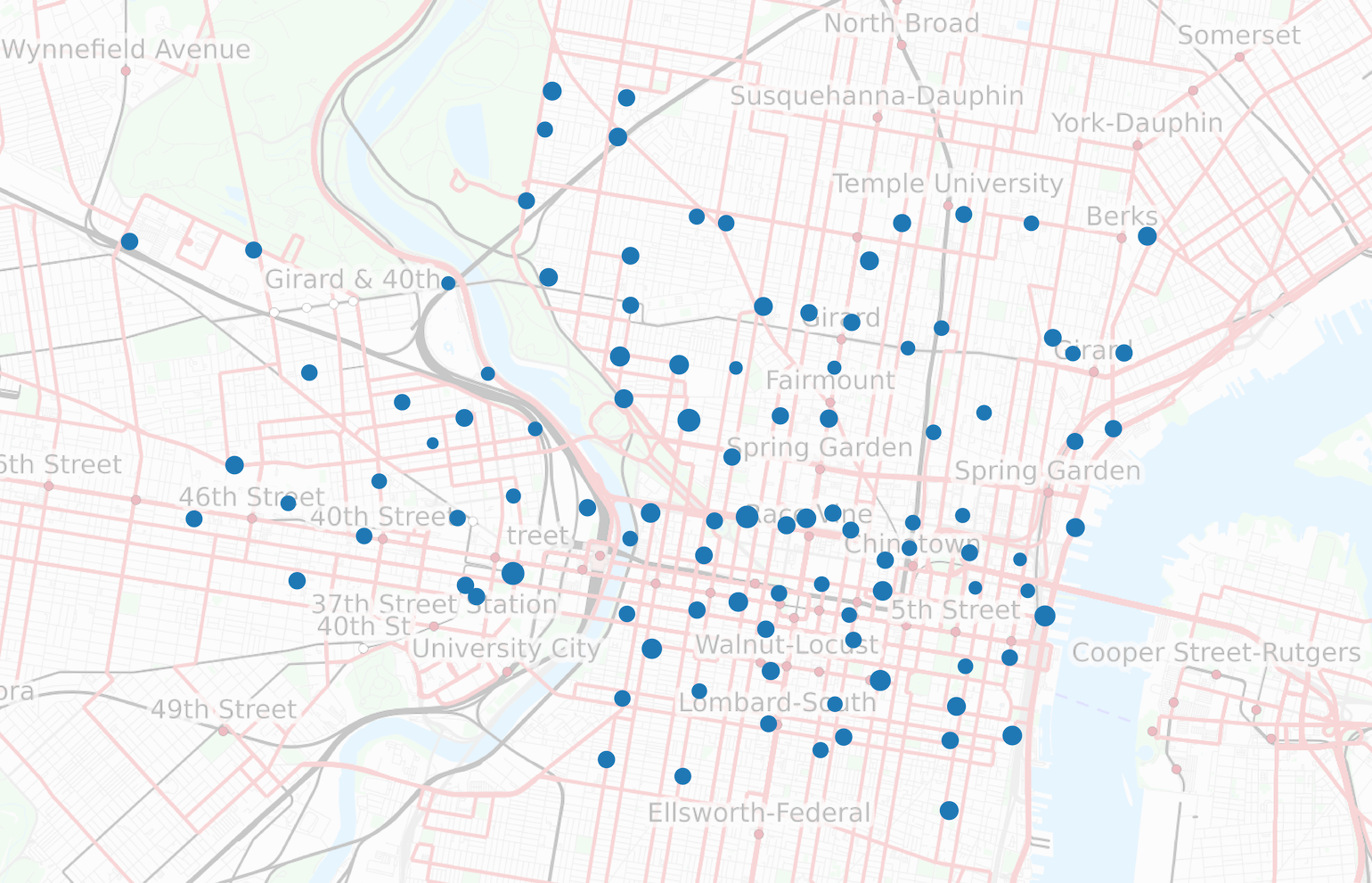}
\caption{Station dataset used in the model of the the Philadelphia bike sharing system. Station capacities $\overline{d}_i$ are indicated by the sizes of the marker. Background image source: OpenStreetMap.}
\label{fig:philly_stations}
\end{center}
\end{figure}

\begin{table}[tp]
\caption{Solution statistics and computation time for the Philadelphia case study (mean $\pm$ std.~dev.).}
\begin{center}
{\footnotesize
\begin{tabular}{l|rrrrr}
\toprule
                                         &  \textbf{M1}      &  \textbf{M2-R}    &  \textbf{M2-HI}   & \textbf{M2-I}     & \textbf{M3} \\
\midrule
Service rate change (\%)                & 2.290 $\pm$ 0.590 & 3.742 $\pm$ 0.877 & 4.538 $\pm$ 0.810 & 3.013 $\pm$ 1.295 & 3.597 $\pm$ 1.485 \\
Cost change (\$)                         & -11.23 $\pm$ 3.19 & -18.20 $\pm$ 4.04 & -20.96 $\pm$ 4.05 & -12.21 $\pm$ 6.17 & -14.03 $\pm$ 6.24 \\
Stage 1 time (s)                         & 138.0 $\pm$ 228.3 & 3.504 $\pm$ 1.173 & 16.61 $\pm$ 30.69 & 14.06 $\pm$ 58.24 & ---${}^\dagger$   \\
Stage 2 time ($\times 10^{-3}$ s)        & 15.73 $\pm$ 5.62  & 21.32 $\pm$ 6.91  & 22.04 $\pm$ 7.10  & 21.38 $\pm$ 6.87  & 19.31 $\pm$ 4.73  \\
Final integer solution time${}^\ast$ (s) & ---${}^\dagger$   & 456.7 $\pm$ 410.3 & 152.7 $\pm$ 158.9 & ---${}^\dagger$   & 1200 $\pm$ 0      \\
\bottomrule
\end{tabular}\\
${}^\ast$ Timeout set to 1200 s. Solution stats reported are for best solution available at timeout.\\[-0.2cm] ${}^\dagger$ Not required in this method.
}
\end{center}
\label{tab:philly_stats}
\end{table}

\section{Conclusion} \label{sec:C}

We described and validated a scalable approach to look-ahead optimization of shared mobility systems in real time, intended for a system operator wishing to minimize a generalized cost metric by redistributing \acp{SV} in parallel with customers using the system. The method is able to handle stochasticity in customer demand and journey valuations. The method offers the flexibility needed for a time-constrained implementation, in that the SPAR iterations can be limited in number, or accelerated by using linear relaxations. We showed that results were relatively robust with respect to stochasticity and the use of integer relaxations, and to different choices of step size rule. In variants M2-R, M2-HI, and M3, where a final integer solution needs to be computed after the SPAR iterations, this computation could also be terminated early if required, as feasible approximate solutions to \eqref{eq:stage1_approx} are trivial for the solver to find. Lastly, the approximate value functions $\overline{V}_i^t(\cdot)$ are useful as an operational indicator of \ac{SV} surplus or shortfall, and could perhaps be used as an input to other, non-solver-based heuristics for computing first-stage decisions (i.e., routing \acp{RV} and loading and transporting \acp{SV}).

The approach we have presented is intended to be used as part of a real-time scheme in which the system operator plans several hours ahead with the aid of a stochastic demand model. After a short time has elapsed, corresponding to one or more discrete time steps, a new optimization would be carried out in light of the new state measurement and updated demand forecast. This ``receding horizon'' control scheme would resemble \ac{MPC} \cite{borrelli_predictive_2017}. An important consideration in \ac{MPC} is the penalty function assigned to the system state at the end of the optimization horizon; this ``terminal cost'' should reflect the value function (in the dynamic programming sense) of an underlying problem over a much longer planning horizon. It may be possible to use the same separable function approximation approach as we used here to derive a suitable terminal costs for an adapted version of problem \eqref{eq:DRRP}.

A future direction with real-world implications is to consider the limit to station-free, also known as free-floating, mobility systems, in which the number of nodes conceptually tends to infinity and uncertainties are described by continuous probability distributions. In this setting one must find an alternative way to map the infinite-dimensional input data to a finite decision problem at acceptable computational cost, identifying a suitable basis for the value function approximation in the process.



%


\section*{Acknowledgments}

J.~Warrington gratefully acknowledges a visiting fellowship funded by the Simons Institute for the Theory of Computing at UC Berkeley, USA, for the Spring Semester of 2018. Most of the work for this study was carried out there. He also thanks Jannik Matuschke of Technische Universit{\" a}t M{\" u}nchen, Germany, for informative conversations on network flow problems, and both authors thank Lioba Heimbach and Christoph Adam at ETH Zurich for data processing efforts in an earlier project related to the Philadelphia case study.

{\small
\bibliographystyle{plain}
\bibliography{warrington}
}



\appendix

\section*{Appendix A: Details of proof of Proposition 1}

The min-cost flow problem is of the form $\min_{x} c^\top x$ subject to $\sum_{m \rightarrow n} x_{m} - \sum_{n \rightarrow p} x_{p} = b_n\,\, \forall n$ and $0 \leq x \leq \overline{x}$, where the scalars $b_n$ and the elements of $\overline{x}_l$ are all integer valued. Vector $x$ takes the form 
\begin{displaymath}
x = \bmat{w^{(1)} \\ d^{(1)}\\ \vdots \\ w^{(T)} \\ d^{(T)}} \quad \text{where} \quad w^{(t)} := \bmat{\underline{w}_{1,1}^{t,0} \\ \underline{w}_{1,1}^{t,1} \\ \vdots \\ \underline{w}_{i,j}^{t,k} \\ \vdots} \quad \text{with each $\underline{w}_{i,j}^{t,k} \in \mathbb{R}^{f_{i,j}^{t,k}(\xi)}$ and} \quad d^{(t)} := \bmat{ d_1^t \\ \vdots \\ d_{|\mathcal{N}_\text{SV}|}^t } \, .
\end{displaymath}
The elements of each subvector $\underline{w}_{i,j}^{t,k}$ correspond to the edge flows shown in Fig.~\ref{fig:s2flowa}. The elements of $\overline{x}_l$ are equal to $1$ for all elements in subvectors $\underline{w}_{i,j}^{t,k}$, and equal to $\overline{d}_i^t$ for element $i$ of $d^{(t)}$. The overall problem's cost vector $c$ has the form 
\begin{displaymath}
c = \bmat{c^{(1)} \\ \mathbf{0} \\ \vdots \\ c^{(T)} \\ \mathbf{0}} \, \text{, where } \, c^{(t)} := \bmat{\overline{l}_{1,1}^{t,0} \\ \overline{l}_{1,1}^{t,1} \\ \vdots \\ \overline{l}_{i,j}^{t,k} \\ \vdots} \, \text{, with each $\overline{l}_{i,j}^{t,k} \in \mathbb{R}^{f_{i,j}^{t,k}(\xi)}$ given by } \, \overline{l}_{i,j}^{t,k} := \bmat{ \overline{l}_{i,j}^{t,k,1} \\ \vdots \\ \overline{l}_{i,j}^{t,k,f_{i,j}^{t,k}(\xi)} } \, \text{ as in Fig.~\ref{fig:s2flowa}}.
\end{displaymath}

There are $|\mathcal{N}_\text{SV}|T + 1$ flow conservation constraints, corresponding to the nodes (plus final sink) shown in Fig.~\ref{fig:s2flowb}. The right-hand side $b_n$ is given by $y_i^{+,t} - y_i^{-,t} - d_i^0$ for $t = 1$ (first column of nodes), and $y_i^{+,1} - y_i^{-,t}$ for $t \geq 2$. The value of the final sink is $\sum_{i \in \mathcal{N}_\text{SV}} [d_i^0 + \sum_{t=1}^T (y_i^{-,t} - y_i^{+,t})]$.

\section*{Appendix B: Results tables}

This appendix contains the full numerical results from which graphs in Section \ref{sec:num_results} are drawn. In tables where several methods are compared, the best value is highlighted in bold.

\begin{table}[htp]
\caption{Comparison of average service rate in final (integer) solutions, for 10 random systems per row; mean $\pm$ standard deviation.}
\begin{center}
{\footnotesize
\begin{tabular}{ll|l|lllll}
\toprule
& & Svc.~rate (\%) & \multicolumn{5}{l}{Increase in service rate (percentage points); mean $\pm$ standard deviation} \\
    &    &  \textbf{NA}      &  \textbf{M1}        & \textbf{M2-R} & \textbf{M2-HI} & \textbf{M2-I} & \textbf{M3}  \\
$|\mathcal{N}_\text{SV}|$ & $|\mathcal{V}|$ &  
	             No action   &  Determ.            & SPAR Relaxed        & SPAR Half Int.     & SPAR Int.         & SPAR Rand.  \\
\midrule
9   & 1  &  80.18 $\pm$ 6.85 &    6.82 $\pm$ 2.94  &    4.98 $\pm$ 3.96  & \bt{7.18} $\pm$ 2.51 &   7.03 $\pm$ 4.50 &   1.37 $\pm$ 2.82 \\
    & 3  &  80.12 $\pm$ 6.82 &    9.08 $\pm$ 3.99  &    4.82 $\pm$ 4.17  & \bt{7.86} $\pm$ 4.15 &   6.64 $\pm$ 5.65 &   4.42 $\pm$ 3.95 \\
16  & 1  &  77.94 $\pm$ 2.94 &    4.69 $\pm$ 1.48  &    3.23 $\pm$ 1.71  & \bt{4.85} $\pm$ 2.22 &   4.68 $\pm$ 2.45 &   0.71 $\pm$ 0.94 \\
    & 5  &  77.69 $\pm$ 2.89 & \bt{8.80} $\pm$ 1.65  &    6.43 $\pm$ 3.19  &    7.83 $\pm$ 1.83 &   8.00 $\pm$ 1.51 &   7.23 $\pm$ 3.31 \\
25  & 1  &  73.93 $\pm$ 4.85 & \bt{3.27} $\pm$ 1.56  &    2.37 $\pm$ 1.93  &    2.53 $\pm$ 1.41 &   2.61 $\pm$ 2.00 &   0.32 $\pm$ 0.76 \\
    & 5  &  73.90 $\pm$ 4.90 & \bt{10.24} $\pm$ 3.27  &    6.80 $\pm$ 2.37  &    9.13 $\pm$ 3.73 &   9.05 $\pm$ 4.27 &   6.41 $\pm$ 2.44 \\
    & 9  &  73.93 $\pm$ 4.83 & \bt{11.58} $\pm$ 3.95  &    7.54 $\pm$ 3.38  &    9.54 $\pm$ 3.30 &   9.94 $\pm$ 3.75 &   9.37 $\pm$ 3.52 \\
36  & 1  &  73.97 $\pm$ 3.06 & \bt{2.45} $\pm$ 1.15  &    1.43 $\pm$ 1.26  &    1.92 $\pm$ 1.32 &   1.87 $\pm$ 1.12 &   0.06 $\pm$ 0.51 \\
    & 5  &  74.05 $\pm$ 3.19 & \bt{8.03} $\pm$ 2.04  &    6.27 $\pm$ 2.12  &    8.14 $\pm$ 2.23 &   6.38 $\pm$ 1.60 &   5.39 $\pm$ 2.06 \\
    & 11 &  73.99 $\pm$ 3.04 &    9.81 $\pm$ 2.45  &    6.65 $\pm$ 1.89  & \bt{10.23} $\pm$ 2.48 &   8.32 $\pm$ 1.71 &   6.05 $\pm$ 1.95 \\
64  & 1  &  72.40 $\pm$ 3.98 & \bt{1.11} $\pm$ 0.62  &    0.11 $\pm$ 0.54  &    0.36 $\pm$ 0.50 &   0.88 $\pm$ 0.47 &   0.02 $\pm$ 0.20 \\
    & 9  &  72.60 $\pm$ 4.10 &    7.49 $\pm$ 2.22  &    6.81 $\pm$ 2.16  & \bt{7.91} $\pm$ 1.98 &   7.37 $\pm$ 2.47 &   4.65 $\pm$ 2.31 \\
    & 15 &  72.43 $\pm$ 4.01 &    8.72 $\pm$ 2.37  &    8.12 $\pm$ 2.85  & \bt{9.61} $\pm$ 2.68 &   8.07 $\pm$ 2.53 &   6.17 $\pm$ 1.91 \\
100 & 1  &  73.69 $\pm$ 1.95 & \bt{0.46} $\pm$ 0.50  &   -0.03 $\pm$ 0.18  &    0.04 $\pm$ 0.32 &   0.16 $\pm$ 0.31 &  -0.04 $\pm$ 0.22 \\
    & 9  &  73.67 $\pm$ 1.96 &    4.90 $\pm$ 0.45  &    3.95 $\pm$ 1.19  & \bt{5.31} $\pm$ 0.46 &   4.83 $\pm$ 0.77 &   0.72 $\pm$ 0.74 \\
    & 19 &  73.67 $\pm$ 2.00 &    7.10 $\pm$ 1.25  &    6.14 $\pm$ 1.21  & \bt{7.83} $\pm$ 1.29 &   6.72 $\pm$ 1.22 &   4.46 $\pm$ 1.23 \\
225 & 1  &  72.47 $\pm$ 2.37 &   -0.01 $\pm$ 0.23  &   -0.00 $\pm$ 0.17  &   -0.05 $\pm$ 0.15 &  -0.09 $\pm$ 0.16 & \bt{0.01} $\pm$ 0.24 \\
    & 13 &  72.45 $\pm$ 2.43 & \bt{3.80} $\pm$ 0.62  &    2.69 $\pm$ 1.23  &    3.35 $\pm$ 0.80 &   3.18 $\pm$ 0.85 &   0.14 $\pm$ 0.43 \\
    & 25 &  72.50 $\pm$ 2.45 & \bt{5.55} $\pm$ 0.92  &    4.81 $\pm$ 0.94  &    5.41 $\pm$ 1.05 &   4.93 $\pm$ 0.76 &   2.51 $\pm$ 0.94 \\
\bottomrule
\end{tabular}
}
\end{center}
\label{tab:method_comparison_svc}
\end{table}%

\begin{table}[htp]
\caption{Comparison of solution cost in final (integer) solution, for 10 random systems per row.}
\begin{center}
{\footnotesize
\begin{tabular}{ll|l|lllll}
\toprule
& & Sol'n cost (\$) & \multicolumn{5}{l}{Absolute change in cost (\$); mean $\pm$ standard deviation} \\
                  &                 &  \textbf{NA} &  \textbf{M1} & \textbf{M2-R} & \textbf{M2-HI} & \textbf{M2-I} & \textbf{M3}  \\
$|\mathcal{N}_\text{SV}|$ & $|\mathcal{V}|$ &  No action   &  Determ.  & SPAR Relaxed & SPAR Half Int. & SPAR Int. & SPAR Rand.  \\
\midrule
9   & 1  &    13.16 $\pm$ 5.35  &     -5.04 $\pm$ 2.26  &    -3.91 $\pm$ 3.06  & \bt{-5.41} $\pm$ 2.17  &     -4.82 $\pm$ 3.20  &    -1.25 $\pm$ 2.23 \\
    & 3  &    13.27 $\pm$ 5.31  & \bt{-6.27} $\pm$ 2.98  &    -3.64 $\pm$ 3.24  &      -5.60 $\pm$ 3.21  &     -4.85 $\pm$ 4.44  &    -2.96 $\pm$ 2.95 \\
16  & 1  &    28.91 $\pm$ 6.42  &     -6.80 $\pm$ 2.40  &    -4.81 $\pm$ 3.17  &  \bt{-7.16} $\pm$ 3.38  &     -6.97 $\pm$ 4.01  &    -1.23 $\pm$ 1.67 \\
    & 5  &    29.28 $\pm$ 6.08  & \bt{-12.53} $\pm$ 3.67  &    -9.31 $\pm$ 4.97  &     -11.13 $\pm$ 2.74  &    -11.01 $\pm$ 3.60  &    -9.93 $\pm$ 5.21 \\
25  & 1  &   52.91 $\pm$ 11.92  & \bt{-7.39} $\pm$ 2.93  &    -4.88 $\pm$ 4.50  &      -5.42 $\pm$ 3.23  &     -5.77 $\pm$ 2.68  &    -0.27 $\pm$ 1.55 \\
    & 5  &   53.18 $\pm$ 12.74  & \bt{-22.49} $\pm$ 7.86  &   -15.64 $\pm$ 6.92  &     -20.09 $\pm$ 8.55  &   -21.08 $\pm$ 11.00  &   -14.09 $\pm$ 5.90 \\
    & 9  &   53.27 $\pm$ 12.18  & \bt{-25.22} $\pm$ 9.65  &   -17.36 $\pm$ 8.64  &     -21.12 $\pm$ 6.50  &    -21.70 $\pm$ 8.56  &   -20.45 $\pm$ 7.71 \\
36  & 1  &   73.05 $\pm$ 14.33  & \bt{-7.49} $\pm$ 3.21  &    -4.19 $\pm$ 4.27  &      -6.03 $\pm$ 3.82  &     -6.39 $\pm$ 3.41  &    -0.56 $\pm$ 1.73 \\
    & 5  &   72.80 $\pm$ 15.01  & \bt{-24.99} $\pm$ 7.97  &   -19.56 $\pm$ 6.62  &     -24.83 $\pm$ 7.48  &    -19.52 $\pm$ 6.21  &   -16.95 $\pm$ 7.37 \\
    & 11 &   72.97 $\pm$ 14.51  &    -29.80 $\pm$ 9.92  &   -20.78 $\pm$ 6.75  & \bt{-31.20} $\pm$ 8.49  &    -25.66 $\pm$ 6.16  &   -17.97 $\pm$ 7.30 \\
64  & 1  &  141.61 $\pm$ 34.72  & \bt{-6.01} $\pm$ 3.78  &    -0.98 $\pm$ 3.29  &      -1.53 $\pm$ 3.43  &     -5.76 $\pm$ 2.59  &    -0.83 $\pm$ 1.31 \\
    & 9  &  139.59 $\pm$ 34.07  &   -41.76 $\pm$ 14.17  &  -38.55 $\pm$ 13.76  & \bt{-43.65} $\pm$ 10.63  &   -42.51 $\pm$ 14.64  &  -26.42 $\pm$ 15.79 \\
    & 15 &  141.13 $\pm$ 34.62  &   -47.45 $\pm$ 13.64  &  -47.10 $\pm$ 20.95  & \bt{-53.42} $\pm$ 15.87  &   -43.89 $\pm$ 15.43  &  -33.89 $\pm$ 13.70 \\
100 & 1  &  199.80 $\pm$ 25.40  & \bt{-3.99} $\pm$ 3.94  &     0.86 $\pm$ 1.57  &      -0.14 $\pm$ 3.20  &     -1.66 $\pm$ 4.14  &     0.29 $\pm$ 1.60 \\
    & 9  &  200.55 $\pm$ 25.51  &    -40.44 $\pm$ 6.31  &  -34.58 $\pm$ 11.42  & \bt{-44.86} $\pm$ 4.85  &    -40.99 $\pm$ 5.73  &    -6.86 $\pm$ 7.75 \\
    & 19 &  200.63 $\pm$ 26.06  &   -58.74 $\pm$ 11.10  &  -52.35 $\pm$ 12.92  & \bt{-65.47} $\pm$ 12.72  &   -56.50 $\pm$ 10.68  &  -37.79 $\pm$ 12.79 \\
225 & 1  &  457.94 $\pm$ 78.91  &     -0.19 $\pm$ 5.00  &    -0.81 $\pm$ 3.81  &       0.52 $\pm$ 2.87  &      1.19 $\pm$ 3.91  & \bt{-0.94} $\pm$ 5.05 \\
    & 13 &  457.91 $\pm$ 78.59  & \bt{-71.34} $\pm$ 17.27  &  -52.62 $\pm$ 25.18  &    -61.70 $\pm$ 18.00  &   -61.55 $\pm$ 20.09  &    -2.28 $\pm$ 7.96 \\
    & 25 &  457.12 $\pm$ 79.86  & \bt{-101.34} $\pm$ 24.82  &  -89.83 $\pm$ 23.30  &   -100.28 $\pm$ 27.08  &   -90.53 $\pm$ 19.27  &  -46.15 $\pm$ 19.25 \\
\bottomrule
\end{tabular}
}
\end{center}
\label{tab:method_comparison_cost}
\end{table}%

\begin{table}[htp]
\caption{Average computation time in seconds per algorithm iteration for M1 and M2. Note Stage 1 is not solved in M3 except in line \ref{algl:final_integer} of Algorithm \ref{alg:spar}. Stage 2 computation times for methods other than M2-I were similar.}
\begin{center}
{\footnotesize
\begin{tabular}{cc|rrrr|r}
\toprule
& & \multicolumn{4}{l}{Stage 1 (line \ref{algl:stage1} of Algorithm \ref{alg:spar})} & Stage 2 \\
    &    &  \textbf{M1} &  \textbf{M2-R} &  \textbf{M2-HI} & \textbf{M2-I} & \textbf{M2-I} \\
$|\mathcal{N}_\text{SV}|$ & $|\mathcal{V}|$ & Determ. &  SPAR Relaxed & SPAR Half Int.    & SPAR Int. & \\
\midrule
9   & 1  &    0.135 &        \bt{0.023} &                0.097 &            0.127 &   0.00155 \\
    & 3  &    0.224 &        \bt{0.032} &                0.230 &            0.457 &   0.00154 \\
16  & 1  &    0.906 &        \bt{0.046} &                0.215 &            0.294 &   0.00228 \\
    & 5  &    1.370 &        \bt{0.079} &                0.734 &            1.986 &   0.00231 \\
25  & 1  &    1.929 &        \bt{0.081} &                0.331 &            0.469 &   0.00407 \\
    & 5  &    3.624 &        \bt{0.145} &                1.890 &            3.823 &   0.00393 \\
    & 9  &    1.654 &        \bt{0.167} &                2.119 &            6.172 &   0.00392 \\
36  & 1  &    1.207 &        \bt{0.139} &                0.419 &            0.587 &   0.00543 \\
    & 5  &    4.121 &        \bt{0.336} &                2.160 &            4.363 &   0.00538 \\
    & 11 &    4.569 &        \bt{0.338} &                5.233 &           10.780 &   0.00538 \\
64  & 1  &    1.275 &        \bt{0.270} &                0.678 &            0.703 &   0.00951 \\
    & 9  &   15.064 &        \bt{0.879} &                7.585 &           14.710 &   0.00912 \\
    & 15 &   14.234 &        \bt{0.868} &               12.371 &           28.281 &   0.00953 \\
100 & 1  &    1.819 &        \bt{0.563} &                1.662 &            1.365 &   0.01407 \\
    & 9  &   27.053 &        \bt{1.448} &               11.446 &           17.399 &   0.01413 \\
    & 19 &   29.242 &        \bt{1.963} &               24.591 &           44.096 &   0.01388 \\
225 & 1  &    7.389 &        \bt{1.728} &                5.078 &            4.871 &   0.05446 \\
    & 13 &  117.257 &        \bt{4.490} &               28.669 &           41.649 &   0.05341 \\
    & 25 &  176.919 &        \bt{5.886} &               64.274 &           96.225 &   0.05060 \\
\bottomrule
\end{tabular}
}
\end{center}
\label{tab:comp_times}
\end{table}

\begin{table}[htp]
\caption{Improvement in service rate from method M2-I over iterations $k$. Values reported are mean $\pm$ standard deviation over 10 random instances. The ``No action'' column reports the baseline service rate, and differ slightly between rows for a given $|\mathcal{N}_\text{SV}|$ due to independent sampling of customer demand.}
\begin{center}
{\footnotesize
\begin{tabular}{ll|l|lll}
\toprule
& & Svc.~rate (\%) & \multicolumn{3}{l}{Increase in service rate (percentage points)}\\
$|\mathcal{N}_\text{SV}|$ & $|\mathcal{V}|$ &  No action   &      $k=10$    &   $k=20$      &     $k=50$     \\
\midrule
9   & 1  &  79.96 $\pm$ 7.10 &  5.96 $\pm$ 4.27 &   6.07 $\pm$ 5.15 &   7.03 $\pm$ 4.50 \\
    & 3  &  80.51 $\pm$ 6.86 &  4.47 $\pm$ 6.17 &   6.18 $\pm$ 4.70 &   6.64 $\pm$ 5.65 \\
16  & 1  &  77.54 $\pm$ 3.16 &  3.84 $\pm$ 2.34 &   3.69 $\pm$ 1.84 &   4.68 $\pm$ 2.45 \\
    & 5  &  77.81 $\pm$ 3.02 &  4.76 $\pm$ 2.30 &   5.99 $\pm$ 2.88 &   8.00 $\pm$ 1.51 \\
25  & 1  &  74.03 $\pm$ 4.81 &  1.90 $\pm$ 2.16 &   2.27 $\pm$ 1.87 &   2.61 $\pm$ 2.00 \\
    & 5  &  74.03 $\pm$ 5.03 &  7.17 $\pm$ 4.68 &   8.23 $\pm$ 4.28 &   9.05 $\pm$ 4.27 \\
    & 9  &  73.89 $\pm$ 4.85 &  7.35 $\pm$ 6.01 &   7.66 $\pm$ 4.33 &   9.94 $\pm$ 3.75 \\
36  & 1  &  74.11 $\pm$ 3.04 &  1.24 $\pm$ 0.67 &   1.51 $\pm$ 0.82 &   1.87 $\pm$ 1.12 \\
    & 5  &  74.01 $\pm$ 2.94 &  4.37 $\pm$ 2.15 &   5.45 $\pm$ 1.13 &   6.38 $\pm$ 1.60 \\
    & 11 &  74.18 $\pm$ 3.01 &  5.22 $\pm$ 2.17 &   6.19 $\pm$ 2.38 &   8.32 $\pm$ 1.71 \\
64  & 1  &  72.47 $\pm$ 4.15 &  0.75 $\pm$ 0.52 &   0.72 $\pm$ 0.64 &   0.88 $\pm$ 0.47 \\
    & 9  &  72.47 $\pm$ 4.03 &  5.37 $\pm$ 2.25 &   6.44 $\pm$ 2.23 &   7.37 $\pm$ 2.47 \\
    & 15 &  72.44 $\pm$ 3.98 &  6.12 $\pm$ 2.44 &   6.81 $\pm$ 2.50 &   8.07 $\pm$ 2.53 \\
100 & 1  &  73.74 $\pm$ 1.95 &  0.13 $\pm$ 0.34 &   0.13 $\pm$ 0.42 &   0.16 $\pm$ 0.31 \\
    & 9  &  73.63 $\pm$ 1.95 &  3.63 $\pm$ 0.49 &   4.10 $\pm$ 0.55 &   4.83 $\pm$ 0.77 \\
    & 19 &  73.72 $\pm$ 1.97 &  4.72 $\pm$ 1.42 &   5.82 $\pm$ 1.13 &   6.72 $\pm$ 1.22 \\
225 & 1  &  72.53 $\pm$ 2.37 &  0.10 $\pm$ 0.16 &  -0.07 $\pm$ 0.20 &  -0.09 $\pm$ 0.16 \\
    & 13 &  72.54 $\pm$ 2.39 &  2.58 $\pm$ 0.91 &   2.81 $\pm$ 0.89 &   3.18 $\pm$ 0.85 \\
    & 25 &  72.47 $\pm$ 2.38 &  3.61 $\pm$ 0.76 &   4.15 $\pm$ 0.53 &   4.93 $\pm$ 0.76 \\
\bottomrule
\end{tabular}
}
\end{center}
\label{tab:artificial_results_int}
\end{table}%

\begin{table}[htp]
\caption{Service rate from method M2-R over iterations $k$. After iteration 50 using a relaxed Stage 1 model, a final integer solution is generated; cf.~M2-I, Table \ref{tab:artificial_results_int}, where each iteration is already integer-feasible.}
\begin{center}
{\footnotesize
\begin{tabular}{ll|l|llll}
\toprule
& & Svc.~rate (\%) & \multicolumn{4}{l}{Increase in service rate (percentage points)}\\
$|\mathcal{N}_\text{SV}|$ & $|\mathcal{V}|$ &  No action   &      $k=10$    &   $k=20$      &     $k=50$      &      Final integer     \\
\midrule
9   & 1  &  80.18 $\pm$ 6.85 &   5.25 $\pm$ 5.40 &  6.55 $\pm$ 7.29 &   8.43 $\pm$ 5.78 &   4.98 $\pm$ 3.96 \\
    & 3  &  80.12 $\pm$ 6.82 &   3.06 $\pm$ 5.24 &  6.68 $\pm$ 6.10 &   8.94 $\pm$ 4.47 &   4.82 $\pm$ 4.17 \\
16  & 1  &  77.94 $\pm$ 2.94 &   5.08 $\pm$ 3.27 &  5.99 $\pm$ 1.97 &   6.20 $\pm$ 2.59 &   3.23 $\pm$ 1.71 \\
    & 5  &  77.69 $\pm$ 2.89 &  -1.67 $\pm$ 3.46 &  6.55 $\pm$ 2.63 &  10.10 $\pm$ 2.78 &   6.43 $\pm$ 3.19 \\
25  & 1  &  73.93 $\pm$ 4.85 &   4.21 $\pm$ 2.90 &  6.10 $\pm$ 2.16 &   6.41 $\pm$ 2.26 &   2.37 $\pm$ 1.93 \\
    & 5  &  73.90 $\pm$ 4.90 &   5.93 $\pm$ 5.18 &  9.89 $\pm$ 4.22 &  12.02 $\pm$ 3.69 &   6.80 $\pm$ 2.37 \\
    & 9  &  73.93 $\pm$ 4.83 &   2.91 $\pm$ 3.28 &  8.93 $\pm$ 3.74 &  12.33 $\pm$ 4.38 &   7.54 $\pm$ 3.38 \\
36  & 1  &  73.97 $\pm$ 3.06 &   3.52 $\pm$ 1.68 &  3.99 $\pm$ 1.17 &   4.60 $\pm$ 1.34 &   1.43 $\pm$ 1.26 \\
    & 5  &  74.05 $\pm$ 3.19 &   5.52 $\pm$ 3.10 &  8.11 $\pm$ 3.21 &  10.31 $\pm$ 2.84 &   6.27 $\pm$ 2.12 \\
    & 11 &  73.99 $\pm$ 3.04 &   0.42 $\pm$ 2.98 &  9.71 $\pm$ 2.64 &  11.46 $\pm$ 2.04 &   6.65 $\pm$ 1.89 \\
64  & 1  &  72.40 $\pm$ 3.98 &   1.90 $\pm$ 1.19 &  2.54 $\pm$ 0.93 &   2.24 $\pm$ 0.59 &   0.11 $\pm$ 0.54 \\
    & 9  &  72.60 $\pm$ 4.10 &   5.75 $\pm$ 3.40 &  9.34 $\pm$ 2.74 &  10.96 $\pm$ 2.59 &   6.81 $\pm$ 2.16 \\
    & 15 &  72.43 $\pm$ 4.01 &   3.12 $\pm$ 4.90 &  9.38 $\pm$ 3.75 &  12.75 $\pm$ 3.38 &   8.12 $\pm$ 2.85 \\
100 & 1  &  73.69 $\pm$ 1.95 &   1.29 $\pm$ 0.33 &  1.43 $\pm$ 0.40 &   1.56 $\pm$ 0.33 &  -0.03 $\pm$ 0.18 \\
    & 9  &  73.67 $\pm$ 1.96 &   4.08 $\pm$ 1.32 &  6.24 $\pm$ 0.97 &   8.15 $\pm$ 0.72 &   3.95 $\pm$ 1.19 \\
    & 19 &  73.67 $\pm$ 2.00 &   1.99 $\pm$ 1.63 &  8.21 $\pm$ 1.40 &  10.40 $\pm$ 1.53 &   6.14 $\pm$ 1.21 \\
200 & 1  &  72.47 $\pm$ 2.37 &   0.53 $\pm$ 0.34 &  0.53 $\pm$ 0.18 &   0.71 $\pm$ 0.17 &  -0.00 $\pm$ 0.17 \\
    & 13 &  72.45 $\pm$ 2.43 &   3.35 $\pm$ 1.23 &  4.78 $\pm$ 1.03 &   5.89 $\pm$ 0.75 &   2.69 $\pm$ 1.23 \\
    & 25 &  72.50 $\pm$ 2.45 &   3.65 $\pm$ 1.31 &  6.57 $\pm$ 1.16 &   8.58 $\pm$ 1.03 &   4.81 $\pm$ 0.94 \\
\bottomrule
\end{tabular}
}
\end{center}
\label{tab:artificial_results}
\end{table}%

\end{document}